\documentclass{amsart}
\usepackage[utf8]{inputenc}
\usepackage{amsmath}
\usepackage{mathtools}
\usepackage{amssymb}
\usepackage{tikz}
\usepackage{graphicx}
\usepackage{float}
\usepackage{framed}
\usepackage[hang,flushmargin]{footmisc}
\usepackage{comment}
\usepackage{ytableau}
\usepackage{hyperref}
\usepackage{cleveref}
\usepackage{comment}
\usepackage{amsthm}
\usepackage{mathtools}
\newtheorem{theorem}{Theorem}[section]
\newtheorem{corollary}[theorem]{Corollary}
\newtheorem{lemma}[theorem]{Lemma}

\newtheorem{proposition}[theorem]{Proposition}

\newtheorem{remark}[theorem]{Remark}

\usepackage{scalerel,stackengine}
\stackMath
\newcommand\reallywidehat[1]{%
\savestack{\tmpbox}{\stretchto{%
  \scaleto{%
    \scalerel*[\widthof{\ensuremath{#1}}]{\kern-.6pt\bigwedge\kern-.6pt}%
    {\rule[-\textheight/2]{1ex}{\textheight}}%WIDTH-LIMITED BIG WEDGE
  }{\textheight}% 
}{0.5ex}}%
\stackon[1pt]{#1}{\tmpbox}%
}
\newcommand{\C}{\mathbb{C}}
\newcommand{\R}{\mathbb{R}}

\newcommand{\Z}{\mathbb{Z}}

\newcommand{\li}[1]{\overline{ #1}}
\newcommand{\Tr}{\mathrm{Tr}}
\newcommand{\p}{\mathfrak{p}}
\newcommand{\g}{\mathfrak{g}}
\newcommand{\kf}{\mathfrak{k}}
\newcommand{\h}{\mathfrak{h}}
\newcommand{\af}{\mathfrak{a}}
\newcommand{\n}{\mathfrak{n}}

\newcommand{\Ad}{\mathrm{Ad}}

\newcommand{\slf}{\mathfrak{sl}}
\newcommand{\gl}{\mathfrak{gl}}
\newcommand{\so}{\mathfrak{so}}
\newcommand{\m}{\mathfrak{m}}
\newcommand{\q}{\mathfrak{q}}
\newcommand{\bfr}{\mathfrak{b}}

\begin{document}
\title{Induced representations from the Heisenberg parabolic subgroup of $SL(n+2,\R)$}
\author{Robin van Haastrecht}

\begin{abstract}
The group $G = SL(n+2,\R)$ has a parabolic subgroup $P = MAN$ whose unipotent radical is a Heisenberg group. We study representations induced from a character of $P$, which can be realized inside of the space $L^2(K/L_0)$ and using the Lie algebra action on the space of $K$-finite vectors we find the points of reducibility and the complementary series.
\end{abstract}



\address{Mathematical Science, Chalmers University of Technology and University of Gothenburg, SE-412 96 Göteborg, Sweden. }

\maketitle
\baselineskip 1.15pc

\section{Introduction}

We call a parabolic subgroup $P = MAN$ of a semisimple Lie group $G$ a \emph{Heisenberg} parabolic subgroup if the unipotent radical $N$ is a Heisenberg group. The study of Heisenberg parabolic subgroups and the associated homogeneous space $G/P$ has applications in analysis \cite{Kab12sln2} and representation theory \cite{BKZ08,Fra22,Wei03,Zha22,Zha23}. In representation theory they often appear in the context of the construction of minimal representations using parabolically induced representations. In this paper we analyze the representations induced from a character of the Heisenberg parabolic subgroup of the group $SL(n+2,\R)$.

Genkai Zhang and collaborators \cite{FHWZ26, Zha22, Zha23} have earlier studied different families of these induced representations by restricting to the maximal compact subgroup $K$. The case $G = SL(n+2,\R)$ has been missing until now. The difficulty is two-fold; first of all, as the representation $\p^{\C}$ is rather large, some transitivity properties of the Weyl group which have been exploited in e.g. \cite{Zha22, Zha23} no longer hold. Furthermore, the torus $A$ is two-dimensional, which requires further analysis.

In this paper, we restrict the induced representation to the maximal compact subgroup $K = SO(n+2)$ and explicitly compute product formulas for $(l,p)$-spherical functions for the associated non-compact pair $(\so(2,n),\so(2) \times \so(n))$. This allows us to describe the action of $\p$ on the $K$-types. We study the following questions.
\begin{itemize}
    \item For which parameters are the representations reducible?
    \item For which parameters are the representations unitarizable?
    \item Which subrepresentations appear at points of reducibility?
    \item Which subrepresentations are unitarizable?
\end{itemize}

We find the complementary series. The 'minimal' representation of $SL(n+2,\R)$ (in the sense of \cite{Tam23}, as $SL(n+2,\R)$ does not have a minimal representation) and solutions to the Heisenberg ultrahyperbolic equation have been found as subrepresentations \cite{Fra22,Kab12,Tam23} at some of the points of reducibility. We give a criterion for when the solutions to the Heisenberg ultrahyperbolic equation are unitarizable.

\subsection{Statement of the results}

The group $SL(n+2,\R)$ has a parabolic subgroup $P$ whose nilpotent radical $N$ is a $2n+1$-dimensional Heisenberg group. Let $P = MAN$ be the Langlands decomposition, then $\m \cong \slf(n,\R)$ and $A \cong \R^2_{>0}$. The group $M$ has four connected components and we denote by $\chi_{\delta_1,\delta_2}$, $\delta_1,\delta_2 \in \Z / 2 \Z$ the characters of $M$. The characters of $A$ are parameterized by $\nu \in (\af^{\C})^* \cong \C^2 $, where the identification is made explicit in Section \ref{subsecind}. We form the principal series representations as unnormalized parabolically induced representations
$$ I(\delta_1,\delta_2,\nu) = Ind_P^G(\chi_{\delta_1, \delta_2} \otimes e^{\nu} \otimes 1) \qquad (\delta \in \Z/2\Z=\{0, 1\}, \nu = (\nu_1,\nu_2) \in \C^2). $$
The parameterization is chosen such that $I(\delta_1,\delta_2,\nu)$ is unitary for $\nu_1 \in n+1+\sqrt{-1} \R$ and $\nu_2 \in \sqrt{-1} \R$, and that $I(0,0,\nu)$ contains the trivial representation as a subrepresentation (quotient) for $\nu=(0,0)$ ($\nu=(2n+2,0)$). We obtain the following theorem for $n \geq 2$ (the case $n = 1$ is treated in Proposition \ref{prop-casesl3}).

\begin{theorem}[See Propositions \ref{firstredots}, \ref{redptsnu1nu2}, \ref{ptsredothchar}, \ref{repirrednu1nu2}, \ref{ireddothchar}]
The representation $I(\delta_1,\delta_2,\nu)$ is reducible if and only if
    \begin{itemize}
        \item $(\delta_1, \delta_2) = (0,0)$ and $\nu_1 \in n+2 \Z$ or $\nu_1 \pm \frac{n+2}{n} \nu_2 \in 4 \Z_{\leq 0} \cup (2n+2) + 4 \Z_{\geq 0}$,
        \item $(\delta_1, \delta_2) = (1,1)$ and $\nu_1 \in n+2 \Z$ or $\nu_1 \pm \frac{n+2}{n} \nu_2 \in (-2 + 4 \Z_{\leq 0}) \cup (2n+4) + 4 \Z_{\geq 0}$,
        \item $(\delta_1,\delta_2) =(1,0)$ and $\nu_1 - \frac{n+2}{n} \nu_2 \in 4 \Z_{\leq 0} \cup (2n+2) + 4 \Z_{\geq 0}$ or $\nu_1 + \frac{n+2}{n} \nu_2 \in (-2 + 4 \Z_{\leq 0}) \cup (2n+4) + 4 \Z_{\geq 0}$,
        \item $(\delta_1,\delta_2) =(0,1)$ and $\nu_1 + \frac{n+2}{n} \nu_2 \in 4 \Z_{\leq 0} \cup (2n+2) + 4 \Z_{\geq 0}$ or $\nu_1 - \frac{n+2}{n} \nu_2 \in (-2 + 4 \Z_{\leq 0}) \cup (2n+4) + 4 \Z_{\geq 0}$.
    \end{itemize}
\end{theorem}

We denote $\widetilde{\nu} = (2n+2 - \nu_1, -\nu_2)$. By the general theory of Knapp and Stein (see e.g. \cite[Chapters VII and VIII]{Kna86}), there exists a meromorphic family $A(\delta_1,\delta_2,\nu)$ of intertwining operators $I(\delta_1,\delta_2,\nu) \rightarrow I(\delta_2,\delta_1,\widetilde{\nu})$. Together with the irreducibility of $I(\delta,\delta,\nu)$ for $\nu_1 \in (n,n+2)$ this immediately gives the complementary series.

\begin{corollary}[see Corollary \ref{complseries}]
There is a complementary series $I(0,0,\nu)$ and $I(1,1,\nu)$ for $\nu_1 \in (n,n+2)$ and $\nu_2 \in \sqrt{-1} \R$.
\end{corollary}

Frahm \cite{Fra22} has found subrepresentations of the parabolically induced representation which in a sense correspond to minimal representations for groups of type $A$ at the points $\nu_2 = \pm \frac{n \nu_1}{n+2}$ as a kernel of a conformally invariant system of second order differential operators. These operators were constructed in \cite{BKZ08}. Kable \cite{Kab12sln2} found subrepresentations of Heisenberg parabolically induced representations as solutions of the Heisenberg ultrahyperbolic equations, but without specifying the character of $M$. See \cite{KubOr25} for the nonlinear cover $\widetilde{SL}(3,\R)$, where the authors explicitly and systematically treat all the representations of $M$. We identify these representations as subrepresentations at the point of reducibility $\nu_1 = n$.

\subsection{Method of the proof}

We study the induced representation by restricting it to the maximal compact subgroup $K = SO(n+2)$ and calculating the action of $\g^{\C} \cong \slf(n+2,\C)$. As the restriction to $K$ is not multiplicity-free, the study of the representation becomes somewhat involved. Further difficulties are that the torus $A$ is two-dimensional and the representation $\p^{\C} = \mathrm{Sym}_{n+2}(\C)_{\mathrm{trace-free}}$ of $K$ is somewhat large, which means that we cannot use a transitive property on the Weyl group as in \cite{Zha22, Zha23}. The first result is Theorem \ref{actiondiffop}. Its proof relies on the analysis of a product formula of three spherical functions, as spherical functions of $\mathrm{Sym}_{n+2}(\C)_{\mathrm{trace-free}}$ can be written as the product of spherical functions of $\C^{n+2}$, which is a (quasi-)minuscule representation of $SO(n+2)$. We will also rely on explicit computations of Harish-Chandra's $c$-function for $l$-spherical functions, as studied by e.g. \cite{Sch84, Shim1990}. We will need the exact values and asymptotics of the $c$-function to calculate the points of reducibility and to prove the representation is irreducible at the remaining points. We use the Knapp-Stein intertwining operator for analyzing points of reducibility as in \cite{FHWZ26}.

\subsection{Related works}

The study of induced representations has a long history, see e.g. \cite{HoweLee99, HT93, Joh90, JW77}. The easiest case is generally the one where the unipotent radical is Abelian, and in the case of $SL(n,\R)$ has been studied in e.g. \cite{HoweLee99}. Representations induced from a parabolic subgroup whose unipotent radical is a Heisenberg group have found applications in models for minimal representations and small representations, see e.g. \cite{Fra22, FWZ25,  Kab12, Kab12sln2, KazhSav90}. The general study of representations induced from a parabolic subgroup whose unipotent radical is a Heisenberg group was recently started by Genkai Zhang in \cite{Zha22} and has been continued by him and others \cite{FHWZ26, FWZ22, Zha23}. This paper studies these representations for the group $SL(n+2,\R)$ and its Heisenberg parabolic subgroup. Note that some subrepresentations have been found in \cite{Fra22, Kab12sln2}; as started in \cite{Zha22}, we make a general study of the induced representation. 

\subsection{Organization}

In Section \ref{SecPrel} we recall some basic facts about the group $G = SL(n+2,\R)$, its Heisenberg parabolic subgroup and induced representations. In Section \ref{SecDec} we decompose the induced representation when restricted to $K$ and obtain product formulas for the $(p,l)$-spherical functions. We also obtain explicit formulas for the $c$-function. In Section \ref{SecLieActE} we calculate the action of a distinguished Lie algebra element $E \in \p$. Finally, in Section \ref{SecRed} we use this action to exactly determine the points of reducibility. Here we need some exact information on coefficients for branching rules of tensor products for $SO(n+2)$; the computations involved are rather technical and can be found in Appendix \ref{AppCoeff}. We analyze unitarizability and subrepresentations.

\subsection*{Acknowledgements}

I am grateful to Jan Frahm and Genkai Zhang for some inspiring discussions. I am also grateful to Toshihisa Kubo for some valuable input and literature references. I used ChatGPT GPT-6 to proofread a draft of the paper; it caught some mistakes regarding the case $n=2$, which before did not have a correct exposition. All the results are due to the author.

\subsection{List of notation}

For the convenience of the reader we add a list of notation used throughout the paper.

\begin{enumerate}
    \item $G = SL(n+2,\R)$ is the group of matrices of determinant $1$, with Lie algebra $\g = \slf(n+2,\R)$ the trace-free matrices over $\R$.
    \item $K = SO(n+2) \subseteq G$ is a maximal compact subgroup and $\g = \kf \oplus \p$ is a Cartan decomposition, where $\kf = \so(n+2)$ and $\p = \mathrm{Sym}_{n+2}(\R)_{\text{trace-free}} \cong (\C^{n+2} \odot \C^{n+2})_{\text{trace-free}}$ the symmetric $(n+2) \times (n+2)$-matrices over $\R$.
    \item $\h = \bigoplus \R D_i$ is an $[\frac{n}{2}] + 1$-dimensional Cartan subalgebra of $\so(n+2,\R)$, where $D_i = E_{2i-1,2i} - E_{2i,2i-1}$.
    \item $\varepsilon_i \in \h^*$ is defined by $\varepsilon_i (D_j) = \sqrt{-1} \delta_{ij}$. We choose a basis of positive simple roots $\{ \alpha_i \}$ as in \cite[Appendix C]{Kna02} and let $\{ \omega_i \}_i$ be the corresponding fundamental weights.
    \item As a $\kf$-representation a highest weight vector for $\p^{\C}$ is $E_{\gamma} = \begin{pmatrix} \begin{pmatrix} 1 & \sqrt{-1} \\ \sqrt{-1} & -1
    \end{pmatrix} & 0 \\ 0 & 0 \end{pmatrix}$ with highest weight $2 \omega_1$.
    \item We write $\lambda = (\lambda_1, \lambda_2) = (\lambda_1 + \lambda_2) \varepsilon_1 + \lambda_2 \varepsilon_2$.
    \item $P = MAN$ is the Heisenberg parabolic subgroup, where $\m \cong \slf(n,\R)$, $\af = \R E \oplus \R F$ and $\n = \n_1 \oplus \n_2$ is a $2n+1$-dimensional Heisenberg group. We have dual elements $E^*$, $F^*$ such that $E^*(E) = 1$ and $E^*(F) = 0$, and similarly for $F^*$. We write $\nu = (\nu_1,\nu_2) = \nu_1 E^* + \nu_2 F^*$.
    \item We write $L = K \cap M$, a group with four components where $\mathfrak{l} \cong \so(n)$ and $L_0 \cong SO(n)$ is the identity component. Furthermore, $L_0 \subseteq L_1 = \{ k \in K \ | \ \Ad(k)E_{\gamma} \in \C E_{\gamma} \} = SO(2) \times SO(n)$.
    \item We denote $w_0 = \mathrm{diag}(1,-1,-1,I_{n-1}) \in K$, then
    $$L = L_0 \cup \mathrm{diag}(-I_2,I_n) L_0 \cup w_0 L_0 \cup \mathrm{diag}(-I_2,I_n) w_0 L_0.$$ Furthermore, for $v^l_{\lambda}$ an $l$-spherical vector we write $v^{l,\pm}_{\lambda} = v^l_{\lambda} \pm w_0 \cdot v^l_{\lambda}$.
    \item $\rho_A = (n+1) E^*$ is the half-sum of the positive roots for $P$.
    \item We write the Cartan decomposition for $\so(2,n) = \so(2) \oplus \so(n) \oplus \q$, a non-compact real form of $\so(n+2)$, with strongly orthogonal roots $\gamma_1 < \gamma_2$, $\h^- = \R \sqrt{-1} h_{\gamma_1} \oplus \R \sqrt{-1} h_{\gamma_2}$ and maximal Abelian subalgebra $\bfr \subseteq \q$. We let $\rho$ be the half-sum of the positive restricted roots for $(\so(2,n),\bfr)$.
    \item We denote $C: \so(n+2,\C) \rightarrow \so(n+2,\C)$ the Cayley transform, connected to the real form $\so(2,n)$, sending $\sqrt{-1} \h^-$ to $\bfr$.
    \item We denote $\lambda_c = (C^{-1})^* \lambda|_{\sqrt{-1} \h^-}$.
    \item We denote the Harish-Chandra $l$-spherical function $\Phi_{\lambda,l}$ for the symmetric pair $(\so(2,n),\so(2) \oplus \so(n))$ with one-dimensional character $\chi_l$, $c(\lambda,l)$ the Harish-Chandra $c$-function and $\phi_{\lambda;p,l}$ $(p,l)$-spherical polynomials. We write $\phi_{\lambda;l} \coloneqq \phi_{\lambda;l,l}$ and $\phi_{\lambda;p,l}^{\pm} = \phi_{\lambda;p,l} \pm w_0 \cdot \phi_{\lambda;p,l}$.
\end{enumerate}

\section{Preliminaries}
\label{SecPrel}

\subsection{The group $SL(n+2,\R)$}

We write $G = SL(n+2,\R)$. $G$ has Lie algebra
$$\g = \{X \in \gl(n+2,\R) \ | \ \Tr(X) = 0 \}$$
with Cartan decomposition $\g = \kf + \p$, where $\kf = \so(n+2,\R)$ and $\p = \{ X \in \gl(n+2,\R) \ | \ X^T = X, \ \Tr(X) = 0 \}$. Note that the complexification $\p^{\C} = \{ X \in \gl(n+2,\C) \ | \ X^T = X, \ \Tr(X) = 0 \} \cong (\C^{n+2} \odot \C^{n+2})_{\mathrm{trace-free}}$, where the action of $SO(n+2)$ is the defining action on $\C^{n+2}$. We have a $K$-invariant inner product $\langle A,B \rangle \mapsto \Tr(AB^*)$.

There is a standard $[\frac{n}{2}] + 1$-dimensional Cartan subalgebra $\h = \bigoplus \R D_i$ in $ \so(n+2, \R)$ where $D_i = E_{2i-1,2i} - E_{2i,2i-1}$. Note that the root system of $\so(n+2,\C)$ is of type $B_{[\frac{n}{2}]+1}$ for $n$ odd or $D_{\frac{n}{2}+1}$ for $n$ even. Let $\varepsilon_i \in \h^*$ be defined by $\varepsilon_i(D_j) = \sqrt{-1} \delta_{ij}$, where $\sqrt{-1}$ is the imaginary unit. Then the roots of $\so(n+2,\C)$ are $\{\pm \varepsilon_i \pm \varepsilon_j \} \cup \{ \pm \varepsilon_i \}$ for $n$ odd and $\{\pm \varepsilon_i \pm \varepsilon_j \ | i<j \}$ for $n$ even. We choose the ordering on $\h^*$ so that $\varepsilon_1$ is highest and a basis of positive simple roots is $\{\varepsilon_1 - \varepsilon_2,\dots,\varepsilon_{[\frac{n}{2}]} - \varepsilon_{[\frac{n}{2}]+1}, \varepsilon_{[\frac{n}{2}]+1}\}$ for $n$ odd and $\{\varepsilon_1 - \varepsilon_2,\dots,\varepsilon_{\frac{n}{2}} - \varepsilon_{\frac{n}{2}+1}, \varepsilon_{\frac{n}{2}} + \varepsilon_{\frac{n}{2}+1}\}$ for $n$ even. We order them $\alpha_1, \dots, \alpha_{[\frac{n}{2}]+1}$ and let the corresponding fundamental weights be $\omega_1,\dots, \omega_{[\frac{n}{2}]+1}$. We see that $\omega_1 = \varepsilon_1$. Then the highest weight of $\p^{\C}$ is $\gamma = 2 \varepsilon_1 = 2 \omega_1$ and a highest weight vector is $E_{\gamma} = \begin{pmatrix} \begin{pmatrix} 1 & \sqrt{-1} \\ \sqrt{-1} & -1 \end{pmatrix} & 0 \\ 0 & 0 \end{pmatrix}$.

\subsection{Heisenberg parabolic}
\label{sec-heis-prel}

We recall the Heisenberg parabolic subgroup from \cite[Table 1]{How08}. It is a parabolic subgroup with Langlands decomposition $MAN$, where $N$ is a Heisenberg group, so the Lie algebra $\n$ has a grading $\n = \n_1 + \n_2$ where $\n_2$ is $1$-dimensional. More explicitly,
$$\n_2 = \R \begin{pmatrix} \begin{pmatrix} 0 & 1 \\ 0 & 0 \end{pmatrix} & 0 \\ 0 & 0
\end{pmatrix}$$
and
$$\n_1 = \mathrm{span}_{\R} \{ \begin{pmatrix} 0 & 0 \\ E_{j2} & 0 \end{pmatrix}, \begin{pmatrix} 0 & E_{1j} \\ 0 & 0 \end{pmatrix}\}_{j=1}^n.$$
Furthermore, $\af = \R E \oplus \R F$ where
$$E = \frac{1}{2}(E_{\gamma} + \li{E_{\gamma}}) = \begin{pmatrix} \begin{pmatrix} 1 & 0 \\ 0 & -1 \end{pmatrix} & 0 \\ 0 & 0 \end{pmatrix}$$
is the element providing the grading on $\n$ and $F = \frac{1}{n+2} \begin{pmatrix} n I_2 & 0 \\ 0 & -2 I_n \end{pmatrix}$. Note that $F$ is normalized such that it acts as $\pm 1$ on $\n_{\pm 1}$. 
Finally, $\m = \begin{pmatrix} 0 & 0 \\ 0 & \slf(n,\R) \end{pmatrix}$ and $M_0 \cong SL(n,\R)$. Defining
$$w_0 =  \mathrm{diag}(1,-1,-1,I_{n-1}).$$
we have $M = Z_K(\mathfrak{a}) M_0 = \begin{pmatrix} \pm I_2 & 0 \\ 0 & SL(n,\R) \end{pmatrix} \cup w_0 \begin{pmatrix} \pm I_2 & 0 \\ 0 & SL(n,\R) \end{pmatrix}$. Note that $M$ has $4$ characters $\chi_{\delta_1, \delta_2}$, where $\chi_{\delta_1, \delta_2}(\begin{pmatrix} \mathrm{diag}(a,b) & 0 \\ 0 & A \end{pmatrix} ) = a^{\delta_1} b^{\delta_2}$. We define the elements $E^*,F^* \in \af^*$ such that $E^*(E) = 1$, $E^*(F) = 0$ and similarly for $F^*$. Observe that for the sum of half-roots $\rho_A = \frac{2 \cdot 1+2n}{2} E^* + \frac{1}{2} \left( n - n \right)F^* = (1 + n)E^*$.

For an analysis of the $K$-types we also define $L = K \cap M = \begin{pmatrix} \pm I_2 & 0 \\ 0 & SO(n) \end{pmatrix} \cup w_0 \begin{pmatrix} \pm I_2 & 0 \\ 0 & SO(n) \end{pmatrix}$, a group with four components and with Lie algebra $\mathfrak{l} \cong \so(n)$. We let $L_0$ be the identity component and $L_1 = \{ k \in K \ | \ \Ad(k)E_{\gamma} \in \C E_{\gamma} \} = \begin{pmatrix} SO(2) & 0 \\ 0 & SO(n) \end{pmatrix}$.

\subsection{Induced representation}
\label{subsecind}

For $\nu = (\nu_1, \nu_2) =  
\nu_1 E^* + \nu_2 F^* \in \af^*$ we study the induced representation
$$ I(\delta_1, \delta_2, \nu) = Ind_P^{G} (\chi_{\delta_1, \delta_2} \otimes e^{\nu} \otimes 1 )$$
defined as the space of measurable functions $f$ such that
$$f(gme^{H}n) = e^{-\nu(H)} \chi_{\delta_1, \delta_2}(m)^{-1} f(g), \quad m \in M, \ H \in \af, \ n \in N$$
and $f|_K \in L^2(K)$. The action of $G$ is by the left-regular action. We note that by the decomposition $G = KMAN$ \cite[Chapter V]{Kna86} the $f$ is entirely determined by $f|_K$ and we see immediately that it is right $L$-invariant; hence we can study the space $L^2(K/L, \chi_{\delta_1, \delta_2})$, the $K$-picture. The action of $G$ will be given by
$$ g \cdot f (k) = e^{-\nu(H(g^{-1}k))} \chi_{\delta_1, \delta_2}(\mu(g^{-1}k))^{-1} f(\kappa(g^{-1} k))$$
for the decomposition $g = \kappa(g) \mu(g) e^{H(g)} n \in KMAN = G$ \cite[Chapter VII]{Kna86}.

\section{Decomposition of $L^2(SO(n+2)/SO(n))$ and $c$-function}
\label{SecDec}

For studying the action of $\g = \kf \oplus \p$
$$ I(\delta_1, \delta_2, \nu) = Ind_P^{G} ( \chi_{\delta_1, \delta_2} \otimes e^{\nu} \otimes 1 )$$
we will look at the $K$-picture, i.e. we will study the action of $\g$ on $L^2(K/L, \chi_{\delta_1, \delta_2})$. First we will determine the $\kf$-action, for which we will need the $K$-decomposition. We first consider the case $n \geq 2$; the case $n = 1$ will be handled separately in Section \ref{SL3case}.

\subsection{Decomposition of $L^2(SO(n+2)/SO(n))$}

We first observe that
$$ L^2(K/L_0) \cong \bigoplus_{\delta_1, \delta_2} L^2(K/L, \chi_{\delta_1, \delta_2}),$$
and that
$$ L^2(K/L_0) \cong \bigoplus_{l \in \Z} L^2(K/L_1, \chi_l),$$
where $\chi_l: L_1 \rightarrow \C$ is a character on the $SO(2)$ factor of $L_1 \cong SO(2) \times SO(n)$ given by
$$\begin{pmatrix} \cos(\theta) & \sin(\theta) \\ - \sin(\theta) & \cos(\theta) \end{pmatrix} \mapsto e^{l \sqrt{-1}  \theta}.$$

We now identify a dominant algebraically integral weight of $\mathfrak{so}(n+2,\C)$ $(\lambda_1 + \lambda_2) \varepsilon_1 + \lambda_2 \varepsilon_2$ with  $\lambda = (\lambda_1, \lambda_2)$. Note that this forces $\lambda_1 \in Z_{\geq 0}$, $\lambda_2 \in \Z$ and for $n \geq 3$ we get $\lambda_2 \in \Z_{\geq 0}$. For $n = 2$ we get $- \frac{\lambda_1}{2} \leq \lambda_2$. Let $V_{\lambda}$ be the representation of $\mathfrak{so}(n+2,\C)$ with highest weight $\lambda$. We use the fact that $\so(n+2)$ is a compact form of $\so(2,n) = \{ \begin{pmatrix} A & \sqrt{-1} C \\ - \sqrt{-1} C^T & D \end{pmatrix} \ | A \in \so(2,\R), C \in M_{2 \times n}(\R), D \in \so(n,\R) \}$ to decompose the space $ L^2(K/L_0) \cong \bigoplus_{l \in \Z} L^2(K/L_1, \chi_l)$ by a theorem of Schlichtkrull \cite{Sch84}.

\begin{lemma}
\label{decompKlemma}
As $K$-representations there is a decomposition
$$L^2(SO(n+2)/SO(2) \times SO(n), \chi_l ) \cong \bigoplus_{\lambda \ | \ \lambda_1 \equiv l \mod 2, \ \lambda_1 + \lambda_2 - |\lambda_2| \geq |l|} V_{\lambda},$$
where an explicit map is given by
$$V_{\lambda} \ni v \mapsto \left( k \mapsto \langle k^{-1} v, v^l_{\lambda} \rangle \right)$$
and $v^l_{\lambda}$ is an $l$-spherical vector in $W_{\lambda}$, unique up to constant.
\end{lemma}

\begin{proof}
For $n \geq 3$ this follows from \cite[Theorem 7.2]{Sch84} by standard arguments. See also \cite[Lemma 2.5]{Zha22} and \cite[Lemma 3.1]{Zha23}. The case $n = 2$ is special as the Lie group is not simple. Then we have fundamental weights $\frac{1}{2}(\varepsilon_1 - \varepsilon_2)$ and $\frac{1}{2}(\varepsilon_1 + \varepsilon_2)$ and the highest weight representations of the Lie algebra are given by:
$$ a \frac{1}{2}(\varepsilon_1 - \varepsilon_2) + b \frac{1}{2}(\varepsilon_1 + \varepsilon_2) = \frac{a+b}{2} \varepsilon_1 + \frac{b - a}{2} \varepsilon_2 = a \varepsilon_1 + \frac{b-a}{2} (\varepsilon_1 + \varepsilon_2),$$
where $a+b$ even. The result follows.
\end{proof}

We observe that $\chi_{\delta_1, \delta_2}(w_0) = (-1)^{\delta_2}$ and $w_0 \cdot v_{\lambda}^l$ is $-l$-spherical, so we have the following further decomposition as $K$-representations.
\begin{equation}
\label{decompL2}
L^2(K/L, \chi_{\delta_1, \delta_2}) \cong \bigoplus V_{\lambda} \otimes \overline{V_{\lambda}^{\chi_{\delta_1, \delta_2}}},
\end{equation}
where $V_{\lambda}^{\chi_{\delta_1, \delta_2}} = \mathrm{span} \{ v^l_{\lambda} + (-1)^{\delta_2} \cdot w_0  \cdot v^{l}_{\lambda} \}_{0 \leq l \leq \lambda_1 + \lambda_2 - |\lambda_2|, l \equiv \lambda_1 \equiv \delta_1 + \delta_2 \mod 2}$.

\begin{remark}
We see that $w_0 v_{\lambda}^0 = (-1)^{\lambda_2} v_{\lambda}^0$. This follows by a direct computation the action of $w_0$ on $v^0_{(2,0)} \in V_{(2,0)}$ and $v^0_{(0,1)} \in V_{(0,1)}$. Note that then $v^0_{\lambda} + (-1)^{\delta_2+\lambda_2} v^{0}_{\lambda}$ can be zero.
\end{remark}

\begin{remark}
\label{findimrep}
We note that $\mathfrak{g}^{\C}$ is a representation of $SL(n+2,\R)$, and by $\mathfrak{n}_{-2} = \R X$ we can see that we can embed $\mathfrak{g}^{\C}$ into $I(1,1,(-2,0))$ by
$$ v \mapsto (g \mapsto \langle g^{-1} v, X \rangle).$$
We note furthermore that other finite-dimensional representations can be embedded into the induced representation. We see that $\C^{n+2}$ the defining representation can be embedded into $I(0,1,(-1, \frac{n}{n+2}))$ by
$$ v \mapsto (g \mapsto \langle g^{-1} v, e_2 \rangle.$$
The dual representation $(\C^{n+2})^*$ can be embedded into $I(1,0,(-1, - \frac{n}{n+2}))$ by
$$ v^* \mapsto (g \mapsto \langle g^{-1} v^*, e_1^* \rangle.$$
We then see that finite-dimensional representations of weight $a \omega_1 + b \omega_{n+1}$ appear in the representation $I(b \mod 2, a \mod 2,(-(a+b), (a-b) \frac{n}{n+2}))$. By an analysis of the infinitesimal character of the induced representation we see that these are all finite-dimensional representations that appear as subrepresentations.
\end{remark}

\subsection{Spherical polynomials}

We need some lemmas on $(L_1,l)$-spherical polynomials. Note that for $V_{\lambda} \subseteq L^2(K / L_1, \chi_l)$ and $p$ such that $|p| \leq \lambda_1$ and $p \equiv \lambda_1 \mod 2$ there exists a unique, up to constants, $(l,p)$-spherical polynomial $\phi_{\lambda;l,p}$ in $V_{\lambda} \subseteq L^2(K / L_1, \chi_l)$. We denote $\phi_{\lambda;l} = \phi_{\lambda;l,l}$. As $\p^{\C} \cong (\C^{n+2} \odot \C^{n+2})_{\mathrm{trace-free}}$, we need to analyze products of spherical functions where one of the factors is the spherical polynomial connected to the
representation of $\C^{n+2}$, with highest weight $\omega_1$. First we realize that
\begin{equation}
\label{decompcptrep}
\C^{n+2} \otimes V_{\lambda} \cong ( V_{\lambda} \oplus) \sum_{w \in W / W_{\varepsilon_1}} V_{\lambda + w \cdot \omega_1 },
\end{equation}
where the $V_{\lambda}$ appears when $n$ is odd and does not when $n$ is even \cite[Exercise VI.1.24, Exercise VI.4.15]{Bour02}. We prove a lemma on the exact decomposition of the spherical functions. Let $e_+ = \frac{\sqrt{2}}{2} (e_1 + \sqrt{-1} e_2) \in \C^{n+2}$ be $v^1_{(1,0)}$ the highest weight unit vector. We also define $e_- = \frac{\sqrt{2}}{2} (e_1 - \sqrt{-1} e_2) = v^{-1}_{(1,0)}$.

We will use the non-compact symmetric pair $(\so(2,n), \so(2) \oplus \so(n))$, the dual of the pair $(\so(n+2),\so(2) \oplus \so(n))$, with Cartan decomposition $\so(2,n) = (\so(2) \oplus \so(n)) \oplus \q$. Then $\gamma_1 = \varepsilon_1 - \varepsilon_2$ and $\gamma_2 = \varepsilon_1 + \varepsilon_2$ will be the strongly orthogonal roots, where $\gamma_1 < \gamma_2$. Then $\sqrt{-1} h_{\gamma_1} = D_1 - D_2$ and $\sqrt{-1} h_{\gamma_2} = D_1 + D_2$. Let $\h^{\C}$ be the complexification of $\h$, $\h^- = \R (D_1 - D_2) \oplus \R(D_1 + D_2)$ and $\h^+ = \bigoplus_{i=3}^{[\frac{n}{2}] + 1} \R D_i$. We let $\bfr = \bigoplus_{i =1,2}\R(X_{\gamma_i} + X_{-\gamma_i})$, then $\bfr \subseteq \q$ is a maximal Abelian subalgebra. Let $C: \so(n+2,\C) \rightarrow \so(n+2,\C)$ be the Cayley transform \cite{Kor65}, so $C(\sqrt{-1} \h^-) = \bfr$ and $C$ fixes $\h^+$. For $\lambda \in (\h^{\C})^*$ we introduce the notation
$$\lambda_c = (C^{-1})^* \lambda|_{\sqrt{-1} \h^-},$$
the pullback of the Cayley transform. We will also need the Harish-Chandra $c$-function $c(\lambda,l)$ for $l$-spherical functions \cite{Shim1990}, where we normalize the $(l,l)$-spherical functions such that $\phi_{\lambda,l}(e) = 1$. We denote $W' = W(B:SO(2,n))$ the Weyl group for $( \mathfrak{so}(2,n),\bfr)$ \cite[Chapter V.3]{Kna86}; by a slight abuse of notation, we let it act in the same way on $\lambda|_{\sqrt{-1}\h^-}$ and $\lambda_{c}$. For $\lambda \in (\h^{\C})^*$ such that $\lambda|_{\h^+} = 0$ we let $W$ act on $\lambda$ in the obvious way. We let $a = m_{\varepsilon_1} = m_{\epsilon_2} =n-2$ be the multiplicity of the restricted root $\varepsilon_1$ (and $\varepsilon_2$) for $(\so(2,n), \bfr)$. We let $\rho$ be the half-sum of the positive restricted roots with multiplicity for the $\bfr$-roots, i.e. $\rho \circ C = \frac{1}{2} ( \varepsilon_1 + \varepsilon_2 + \varepsilon_1 - \varepsilon_2 + a \varepsilon_1 + a \varepsilon_2) = \frac{a+2}{2} \varepsilon_1 + \frac{a}{2} \varepsilon_2$.

\begin{proposition}
\label{tensproddec}
For given choices of $v_{\lambda}^l$ we have a decomposition
$$ e_+ \otimes v^l_{\lambda} = \sum_{w \in W' / W'_{\omega_1}} C^1(w \cdot \omega_1,\lambda,l) v_{\lambda+ w \cdot \omega_1}^{l+1},$$
where $|C^1(w \cdot \omega_1,\lambda,l)|^2 = A_{w,\lambda_c + \rho,l}$, $A_{w,\lambda_c,l} = A_{1, w^{-1} \lambda_c,l}$ and $ A_{1,\lambda_c, l} = \frac{c(\lambda_c,l)}{4 c( (\lambda + \omega_1)_c, l+1)}$. Similarly,
$$ e_- \otimes v^l_{\lambda} = \sum_{w \in W' / W'_{\omega_1}} C^{-1}(w \cdot \omega_1,\lambda,l) v_{\lambda+ w \cdot \omega_1}^{l-1},$$
where
$$|C^{-1}(w \cdot \omega_1,\lambda,l)|^2 = A_{w,\lambda_c + \rho,l}^- = A_{1,w^{-1} (\lambda_c + \rho),-l}$$
and
$$A_{1,\lambda_c,-l} = \frac{c(  \lambda_c,l)}{4 c((\lambda + \omega_1)_c, l-1)}.$$
\end{proposition}

\begin{proof}
Our proof will rely on an analysis of product formulas as in \cite{Vret76}, see also \cite{Zha22,Zha23}.  We first realize from Decomposition \ref{decompcptrep} that as functions of $K$
$$ \langle k^{-1} \cdot e^+, e^+ \rangle \langle k^{-1} \cdot v^l_{\lambda}, v_{\lambda}^l \rangle = \phi_{\omega_1;1}(k) \phi_{\lambda;l}(k) = \sum_{w \in W' / W'_{\varepsilon_1}} A_{w,\lambda,l} \phi_{\lambda + w \cdot \omega_1; l+1}(k).$$
We now determine the $A_{w,\lambda,l}$ using the Harish-Chandra $l$-spherical functions. We first notice that, like in \cite{Vret76,Zha22,Zha23}, for $k \in SO(2) \times SO(n)$
$$\phi_{\lambda;l}(k) = \Phi_{\lambda_c + \rho,l}(k),$$
where $\Phi_{\lambda_c + \rho,l}$ is the $l$-spherical function associated to the non-compact pair $(\so(2,n), \so(2) \oplus \so(n))$ in the notation of \cite{Shim1990}. Furthermore, for $H \in \bfr$ it holds \cite[Equation 6.10]{Shim1990},
$$ \lim_{H \rightarrow \infty} e^{(- \lambda + \rho)(\exp(H))} \Phi_{\lambda, l}(\exp(H)) = c(\lambda,l).$$
A direct calculation gives
$$ \langle \exp(H) e_+,e_+ \rangle = \frac{1}{4} e^{(\omega_1)_c(H)} + L.O.T.$$
the $L.O.T.$ being lower order terms. Then for regular $\lambda$ \cite{Vret76}
$$\phi_{\omega_1;1} \Phi_{\lambda,l} = \sum_{w \in W'  / W'_{\omega_1}} A_{w,\lambda,l} \Phi_{\lambda+w \cdot (\omega_1)_c,l+1},$$
where the coefficient $A_{1,\lambda, l} = \frac{c(\lambda,l)}{4 c(\lambda + (\omega_1)_c,l+1)}$. Furthermore,
$$ \phi_{\omega_1;1} \Phi_{\lambda,l} = \phi_{\omega_1,1} \Phi_{w' \cdot \lambda,l} = \sum_{w \in W' / W'_{\varepsilon_1}} A_{w,w' \cdot \lambda,l} \Phi_{w' \cdot \lambda+w \cdot (\omega_1)_c,l+1}.$$
The uniqueness of the coefficients gives $A_{w,\lambda,l} = A_{1,w^{-1} \cdot \lambda,l} \geq 0$. It follows that
\begin{flalign*}
& \langle k^{-1} e_+,e_+ \rangle \phi_{\lambda;l} = \phi_{\omega_1;1} \Phi_{\lambda_c + \rho,l} = \sum_{w \in W' / W'_{\omega_1}} A_{w,\lambda_c +\rho,l} \Phi_{ (\lambda+ w \cdot \omega_1)_c + \rho,l}
\\ & = \sum_{w \in W' / W'_{\omega_1}} A_{w,\lambda_c + \rho,l} \phi_{\lambda+ w \cdot \omega_1;l},
\end{flalign*}
and $A_{w,\lambda + \rho, l} = A_{1,w^{-1} \cdot (\lambda + \rho),l}$. The case $A_{w,\lambda,l}^-$ follows similarly.
\end{proof}

We remark that $W \cdot \omega_1 = \{(1,0),(-1,0),(-1,1),(1,-1) \}$. As we will be dealing with spherical functions of the representation $\p^{\C} \cong \C^{n+2} \odot \C^{n+2}$, we write down the corresponding product formula.

\begin{corollary}
\label{multformsph}
There are constants $C^{i}(\sigma,\lambda,l)$ such that
$$ \langle k^{-1} E_{\gamma}, E_{\gamma} \rangle \phi_{\lambda;p,l}(k) = \sum_{\sigma \in S} C^1(\sigma,\lambda,p) C^1(\sigma,\lambda,l) \phi_{\lambda + \sigma;p+2, l + 2}(k).$$
Here $S = \{(0,0), \pm (2,-1), \pm (2,0), \pm (2,-2), \pm (0,1) \}$. Similarly if we replace some or both of the $E_{\gamma}$ by $\li{E_{\gamma}}$ or one of the $E_{\gamma}$ by $F$ and we write $C^i(\sigma,\lambda,p) C^j(\sigma,\lambda,l) \phi_{\lambda + \sigma;p+2 i, l + 2j}$ for suitable $i = 0, \pm 1, j = 0, \pm 1$.

Furthermore, for $w \in W'$ we have $|C^1(2 w \cdot \omega_1,\lambda,l)|^2 = A_{w,\lambda_c + \rho,l} A_{w,(\lambda + w \cdot \omega_1)_c + \rho,l+1}$. Similarly for $|C^{-1}(2 w \cdot \omega_1,\lambda,l)|$ and $|C^{0}(2 w \cdot \omega_1,\lambda,l)|$.
\end{corollary}

\begin{proof}
We first realize that $\frac{1}{2} E_{\gamma} = e_+ \odot e_+$, and thus 
$$\langle k^{-1} E_{\gamma}, E_{\gamma} \rangle \phi_{\lambda;p,l}(k) = \phi_{\omega_1;1}(k)^2 \phi_{\lambda;p,l}(k).$$
Using Proposition \ref{tensproddec} the result follows for $E_{\gamma}, \li{E_{\gamma}}$. Using the fact that
$$e_+ \otimes e_- = \frac{1}{2} \left( \begin{pmatrix} 0 & - \sqrt{-1} \\ \sqrt{-1} & 0 \end{pmatrix}, 0 \right) \oplus \frac{1}{2} F \oplus \frac{1}{n+2} I_{n+2} \in \kf^{\C} \oplus \p^{\C} \oplus \mathrm{Triv},$$
we see that $F = 2 P_{\p^{\C}}(e_+ \otimes e_-)$ and the result follows for $F$.
\end{proof}

\subsection{Calculating the $c$-function}
\label{cfctncalc}

We note that $W' / W'_{\varepsilon_1}$ has four elements, $e,w_1,w_2,w_3$ where $w_1 \cdot \varepsilon_1 = -\varepsilon_2$, $w_2 \cdot \varepsilon_1 = \varepsilon_2$ and $w_3 \cdot \varepsilon_1 = - \varepsilon_1$. We identify elements in $W'$ with $W' / W'_{\varepsilon_1}$. We will need some explicit expressions for the $A_{w,\lambda + \rho, l}$.

\begin{proposition}
\label{explicitcfctn}
We have
\begin{flalign*}
& A_{1,\lambda_c + \rho,l} = \frac{a + 1 + \lambda_1 + \lambda_2}{4 \lambda_1 + 4 \lambda_2 + 2a + 4} \cdot \frac{2 + \lambda_1 + 2 \lambda_2 + a + l}{\lambda_1 + 2 \lambda_2 + a + 1} \cdot \frac{\lambda_1 + 2 + l}{\lambda_1 + 1},
\\ & A_{w_1,\lambda_c + \rho,l}  = \frac{\lambda_2}{4 \lambda_2 + 2a } \cdot \frac{ \lambda_1 + 2 \lambda_2 + a - l}{ \lambda_1 + 2 \lambda_2 + a + 1} \cdot \frac{2 + \lambda_1 + l}{\lambda_1 + 1},
\\ & A_{w_2,\lambda_c + \rho,l} = \frac{\lambda_2 + a}{4 \lambda_2 + 2a } \cdot \frac{ \lambda_1 + 2 \lambda_2 + a + 2 + l}{ \lambda_1 + 2 \lambda_2 + a + 1} \cdot \frac{\lambda_1 - l}{\lambda_1 + 1},
\\ & A_{w_3,\lambda_c + \rho,l} = \frac{\lambda_1 + \lambda_2 + 1}{ 4 \lambda_1 + 4 \lambda_2 + 2a +4 } \cdot \frac{ \lambda_1 + 2 \lambda_2 + a - l}{ \lambda_1 + 2 \lambda_2 + a + 1} \cdot \frac{\lambda_1 - l}{\lambda_1 + 1}.
\end{flalign*}
\end{proposition}

\begin{proof}
Let $\alpha^* = \frac{ \alpha}{ \langle \alpha, \alpha \rangle_0}$, where $\langle -, - \rangle_0$ denotes the Euclidean inner product on the restricted root space normalized by $ \langle \varepsilon_i,\varepsilon_j \rangle _0=\delta_{ij}$. We see from \cite{Shim1990}
\begin{flalign*}
& A_{1,\lambda_c, l} = \frac{c( \lambda_c,l)}{4 c( (\lambda+\omega_1)_c,l+1)}
\\ & = \frac{1}{4} \prod_{\alpha = \varepsilon_1,\varepsilon_2} \frac{2^{\langle \omega_1, \alpha^* \rangle_0 } \Gamma \left( \langle \lambda, \alpha^* \rangle_0) \Gamma( \frac{1}{2}( \frac{1}{2} m_{\alpha} + 1 + \langle \lambda, \alpha^* \rangle_0 + \langle \omega_1, \alpha^* \rangle_0) \right) \Gamma \left( \frac{1}{2} ( \frac{1}{2} m_{\alpha} + \langle \lambda, \alpha^* \rangle_0 + \langle \omega_1, \alpha^* \rangle_0 ) \right) }{\Gamma ( \langle \lambda, \alpha^* \rangle_0 + \langle \omega_1, \alpha^* \rangle_0 ) \Gamma \left( \frac{1}{2}( \frac{1}{2} m_{\alpha} + 1 + \langle \lambda, \alpha^* \rangle_0 ) \right) \Gamma \left( \frac{1}{2} ( \frac{1}{2} m_{\alpha} + \langle \lambda, \alpha^* \rangle_0 ) \right)}
\\ & \times \prod_{\alpha = \frac{\gamma_1}{2}, \frac{\gamma_2
}{2}} \frac{2^{ \langle \omega_1, \alpha^* \rangle_0} \Gamma( \langle \lambda, \alpha^* \rangle_0 ) \Gamma \left( \frac{1}{2} ( \langle \lambda, \alpha^* \rangle_0 + \langle \omega_1, \alpha^* \rangle_0 + l + 2) \right) \Gamma \left( \frac{1}{2} ( \langle \lambda, \alpha^* \rangle_0 + \langle \omega_1, \alpha^* \rangle_0 - l) \right) }{\Gamma( \langle \lambda, \alpha^* \rangle_0 + \langle \omega_1, \alpha^* \rangle_0 ) \Gamma \left( \frac{1}{2} ( 1 + \langle \lambda, \alpha^* \rangle_0 +  l) \right) \Gamma \left( \frac{1}{2} ( 1 + \langle \lambda, \alpha^* \rangle_0 - l) \right)}
\\ & = 2 \frac{ \Gamma \left( \frac{1}{2}( \frac{1}{2} m_{\varepsilon_1} + 1 + \langle \lambda, \varepsilon_1 \rangle_0 + 1 ) \right) \Gamma \left( \frac{1}{2} ( \frac{1}{2} m_{\varepsilon_1} + \langle \lambda, \varepsilon_1^* \rangle_0 + 1 ) \right) }{ \langle \lambda, \varepsilon_1^* \rangle_0 \Gamma \left( \frac{1}{2}( \frac{1}{2} m_{\epsilon_1} + 1 + \langle \lambda, \varepsilon_1^* \rangle_0) \right) \Gamma \left( \frac{1}{2} ( \frac{1}{2} m_{\varepsilon_1} + \langle \lambda, \varepsilon_1^* \rangle_0 ) \right)}
\\ & \times \prod_{\alpha = \frac{\gamma_1}{2}, \frac{\gamma_2}{2}} \frac{ \frac{1}{2} ( 1 + \langle \lambda, \alpha^* \rangle_0 +  l)  }{ \langle \lambda, \alpha^* \rangle_0 }
\\ & = \frac{ \frac{1}{2} m_{\varepsilon_1} + \langle \lambda, \varepsilon_1 \rangle_0 }{ \langle \lambda, \varepsilon_1^* \rangle_0 } \prod_{\alpha = \frac{\gamma_1}{2}, \frac{\gamma_2}{2}} \frac{ \frac{1}{2} ( 1 + \langle \lambda, \alpha^* \rangle_0 +  l)  }{ \langle \lambda, \alpha^* \rangle_0 }.
\end{flalign*}
We now replace $\lambda_c$ by $\lambda_c + \rho$ to obtain the desired formulas. Using
$$ (\lambda_c + \rho) \circ C = (\lambda_1 + \lambda_2 + \frac{a+2}{2} ) \varepsilon_1 + ( \lambda_2 + \frac{a}{2})\varepsilon_2$$
we then see that, writing $a = m_{\varepsilon_1}$

\begin{flalign*}
& A_{1,\lambda_c + \rho,l} = \frac{\frac{a}{2} + \lambda_1 + \lambda_2 + \frac{a+2}{2}}{\lambda_1 + \lambda_2 + \frac{a+2}{2}} \cdot \frac{ \frac{1}{2} (1 + \lambda_1 + \lambda_2 + \frac{a+2}{2} + \lambda_2 + \frac{a}{2} + l)}{\lambda_1 + \lambda_2 + \frac{a+2}{2} + \lambda_2 + \frac{a}{2}} \cdot \frac{\frac{1}{2}(1 + \lambda_1 + \lambda_2 + \frac{a+2}{2} - \lambda_2 - \frac{a}{2} + l)}{\lambda_1 + \lambda_2 + \frac{a+2}{2} - \lambda_2 - \frac{a}{2}}
\\ & = \frac{a + 1 + \lambda_1 + \lambda_2}{4 \lambda_1 + 4 \lambda_2 + 2a + 4} \cdot \frac{2 + \lambda_1 + 2 \lambda_2 + a + l}{\lambda_1 + 2 \lambda_2 + a + 1} \cdot \frac{\lambda_1 + 2 + l}{\lambda_1 + 1} > 0.
\end{flalign*}
The formulas for $w_1$, $w_2$ and $w_3$ follow similarly.
\end{proof}

\begin{remark}
In the case $n = 2$ we have $a = n - 2 = 0$, which means the denominators in the formulas of Proposition \ref{explicitcfctn} could be $0$ if $\lambda_2 = 0$. This has to do with the fact that the root system of $SO(4,\C)$ is $A_1 \times A_1$. The formulas still hold if we substitute $a = 0$, then take the limit of $\lambda_2$ going to $0$. Alternatively, we see we do not need the first factor $ \frac{ \frac{1}{2} m_{\varepsilon_1} + \langle \lambda, \varepsilon_1 \rangle_0 }{ \langle \lambda, \varepsilon_1^* \rangle_0 }$ in the proof as there are no short roots.
\end{remark}

\section{Lie algebra action of $E$}
\label{SecLieActE}

We know the action of $\kf = \so(n+2)$ on $I(\delta_1,\delta_2,\nu)$ by the $K$-picture. We now determine the action of $E \in \p$; as $\p$ is an irreducible $\kf$-representation, this will yield the whole action of $\g$ and the action of $E$ will suffice to characterize points of reducibility.

\subsection{Reduction to terms}
On smooth functions we denote the Lie algebra action $X$ on $f$ by right-differentiation, i.e. $Xf(g) = \frac{d}{dt}|_{t=0} f(g e^{tX})$

We now calculate the action of $E$ on $(l,p)$-spherical functions $\phi_{\lambda;l,p}$ such that
$$\phi_{\lambda;l,p}(l_1 k l_2) = \chi_l(l_1)^{-1} \phi_{\lambda;l,p}(k) \chi_p(l_2)^{-1}.$$
We note that
\begin{flalign*}
& \pi_{\delta_1,\delta_2,\nu}(E) \phi_{\lambda;l,p}(k) = \frac{d}{dt}|_{t=0} \phi_{\lambda;l,p}(e^{-tE} k) = \frac{d}{dt}|_{t=0} \phi_{\lambda;l,p}(k e^{-t \Ad(k)^{-1}E} ) = -(\Ad(k)^{-1}E) \phi_{\lambda;l,p}(k).
\end{flalign*}
We will find a decomposition of $\Ad(k)E$ according to $(\kf + m) \oplus \af \oplus \n$ modulo $\m \oplus \n$, as we know $\m \oplus \n$ acts trivially from the right. We will need the elements
$$u_{ij} = \begin{pmatrix} 0 & E_{ij} \\ E_{ij}^T & 0 \end{pmatrix} \in \p, z_{ij} = [E,u_{ij}] = \begin{pmatrix} 0 & (-1)^{i+1} E_{ij} \\ (-1)^i E_{ij}^T & 0 \end{pmatrix} \in \kf,$$
defined for $i=1,2$, $j=1,\dots,n$. We note that $\{ u_{ij} + z_{ij} \}_{i,j}$ forms a base for $\n_1$. For any $X \in \p$
\begin{flalign*}
& X \equiv \frac{1}{2} \langle X, E \rangle E + \frac{n+2}{2n} \langle X, F \rangle F + \frac{1}{2} \langle X, ( \begin{pmatrix} 0 & 1 \\ 1 & 0 \end{pmatrix}, 0 ) \rangle ( \begin{pmatrix} 0 & -1 \\ 1 & 0 \end{pmatrix},0) + \sum_{i,j} \frac{1}{2} \langle X, u_{ij} \rangle u_{ij} \mod \m \oplus \n
\\ & \equiv \frac{1}{2} \langle X, E \rangle E + \frac{n+2}{2n} \langle X, F \rangle F + \frac{1}{2} \langle X, ( \begin{pmatrix} 0 & 1 \\ 1 & 0 \end{pmatrix}, 0 ) \rangle ( \begin{pmatrix} 0 & -1 \\ 1 & 0 \end{pmatrix},0) - \sum_{i,j} \frac{1}{2} \langle X, u_{ij} \rangle z_{ij} \mod \m \oplus \n.
\end{flalign*}

We split up $\pi_{\delta_1,\delta_2,\nu}(E) \phi_{\lambda;l,p}(k) = -(\Ad(k)^{-1}E) \phi_{\lambda;l,p}(k)$ according to this decomposition and write
$$- (\Ad(k)^{-1}E) \phi_{\lambda;l,p}(k) = I + II + III + IV.$$

\subsection{Action of $\pi_{\delta_1,\delta_2,\nu}(E)$}

We calculate the action of  $\pi_{\delta_1,\delta_2,\nu}(E)$ on $(p,l)$-spherical functions.

\begin{theorem}
\label{actiondiffop}
If we act by $E$ on $\phi_{\lambda;p,l}$ we obtain
\begin{flalign*}
& \pi_{\delta_1,\delta_2,\nu}(E) \phi_{\lambda;p,l} = \sum_{i,j = \pm 1} \sum_{\sigma} C^i(\sigma,\lambda,p) C^j(\sigma, \lambda,l) \times 
\\ & \left( \frac{\nu_1}{2} - \frac{1}{2} j l + \frac{1}{2} (C(\lambda + \sigma) - C(2 \omega_1) - C(\lambda)) \right) \phi_{\lambda + \sigma; p+2i,l+2j}
\\ & + \frac{\nu_2(n+2)}{n} \sum_{i= \pm 1} \sum_{\sigma} C^i(\sigma,\lambda,p) C^0(\sigma,\lambda,l) \phi_{\lambda+\sigma;p+2i,l}.
\end{flalign*}
\end{theorem}

\begin{proof}
We will heavily rely on Corollary \ref{multformsph} for the proof of this theorem. We first look at the contribution from $I$.
\begin{flalign*}
& - \frac{1}{2} \langle \Ad(k)^{-1}E, E \rangle (E \phi_{\lambda;p,l})(k) = \frac{\nu_1}{8} \langle \Ad(k) (E_{\gamma} + \overline{E_{\gamma}}), E_{\gamma} + \overline{E_{\gamma}} \rangle \phi_{\lambda;p,l}
\\ & = \frac{\nu_1}{2} \sum_{i,j = \pm 1} \sum_{\sigma} C^{i}(\sigma,\lambda,p) C^{j}(\sigma,\lambda,l) \phi_{\lambda + \sigma; p + 2i, l + 2j}.
\end{flalign*}

Secondly we consider the contribution from $II$:
\begin{flalign*}
& - \frac{n+2}{2n} \langle \Ad(k)^{-1} E, F \rangle (F \phi_{\lambda;p,l})(k) = \frac{(n+2) \nu_2}{2n} \langle \Ad(k)^{-1} (E_{\gamma} + \overline{E_{\gamma}}), e_+ \otimes e_- \rangle \phi_{\lambda;p,l}(k)
\\ & = \frac{(n+2) \nu_2}{ n} \sum_{i = \pm 1} \sum_{\sigma} C^{i}(\sigma,\lambda,p) C^0(\sigma,\lambda,l) \phi_{\lambda+\sigma;p+2i,l}(k).
\end{flalign*}

Thirdly, we look at the contribution from $III$. We see that
\begin{flalign*}
& - \frac{1}{2} \langle \Ad(k)^{-1}E, \mathrm{diag} ( \begin{pmatrix} 0 & 1 \\ 1 & 0 \end{pmatrix}, 0 ) \rangle \mathrm{diag} ( \begin{pmatrix} 0 & -1 \\ 1 & 0 \end{pmatrix}, 0) \phi_{\lambda;p,l}
\\ & = - (- \sqrt{-1} l) \frac{1}{8 \sqrt{-1}} \langle \Ad(k)^{-1} (E_{\gamma} + \overline{E_{\gamma}}), E_{\gamma} - \overline{E_{\gamma}} \rangle \phi_{\lambda;p,l} = \frac{l}{8} \langle \Ad(k)^{-1} (E_{\gamma} + \overline{E_{\gamma}}), E_{\gamma} - \overline{E_{\gamma}} \rangle \phi_{\lambda;p,l}
\\ & = \frac{l}{2} \sum_{i,j = \pm 1} \sum_{\sigma} j C^{i}(\sigma,\lambda,p) C^{j}(\sigma,\lambda,l) \phi_{\lambda + \sigma; p + 2i, l + 2j}.
\end{flalign*}

Finally, we look at the contribution from $IV$. We use similar methods as in \cite{BOO96, Zha23}. We set $\phi(k) = \langle \Ad(k)^{-1} E,E \rangle$. We see that
\begin{flalign*}
& \frac{1}{2} \sum_{i,j} \langle \Ad(k)^{-1} E, u_{ij} \rangle z_{ij} \phi_{\lambda;p,l} = - \frac{1}{2} \sum_{i,j} (z_{ij} \phi) (z_{ij} \phi_{\lambda;p,l})
\\ & = - \frac{1}{4} \left( \sum_{i,j} z_{ij}^2 (\phi \phi_{\lambda;p,l}) - z_{ij}^2 (\phi) \cdot \phi_{\lambda;p,l} - \phi \cdot z_{ij}^2 (\phi_{\lambda;p,l}) \right)
\end{flalign*}
Let $D = \sum_{ij} z_{ij}^2$, and let $\{v_i\}$ be an orthonormal basis for $\mathfrak{l}$, the Lie algebra of $L$, and let $H = \frac{\sqrt{2}}{2} \mathrm{diag} ( \begin{pmatrix} 0 & -1 \\ 1 & 0 \end{pmatrix}, 0)$. Let $C$ be the Casimir element, then
$$ C = - \frac{1}{2} D - \sum_i v_i^2 - H^2.$$
This implies that
\begin{flalign*}
& \frac{1}{2} \sum_{i,j} \langle \Ad(k)^{-1} E, u_{i j} \rangle z_{ij} \phi_{\lambda;p,l} = - \frac{1}{4} \sum_{i,j} z_{ij}^2 (\phi \phi_{\lambda;p,l}) - z_{ij}^2 (\phi) \cdot \phi_{\lambda;p,l} - \phi \cdot z_{ij}^2 (\phi_{\lambda;p,l})
\\ & = \frac{1}{2} \left( (C + \sum_i v_i^2 + H^2) (\phi \phi_{\lambda;p,l}) - (C + \sum_i v_i^2 + H^2)(\phi) \cdot \phi_{\lambda;p,l} - \phi \cdot (C + \sum_i v_i^2 + H^2)(\phi_{\lambda;p,l}) \right)
\\ &  = \frac{1}{2} \sum_{i,j = \pm 1} (C - \frac{(l+2j)^2}{2} ) (\phi_{2 \omega_1;2i,2j} \phi_{\lambda;p,l}) - (C - 2)(\phi_{2 \omega_1;2i,2j}) \cdot \phi_{\lambda;p,l} - \phi_{2 \omega_1;2i,2j} (C - \frac{l^2}{2}) (\phi_{\lambda;p,l})
\\ & = \frac{1}{2} \sum_{i,j = \pm 1} \sum_{\sigma} C^i(\sigma,\lambda,p) C^j(\sigma,\lambda,l) \times
\\ & \left( (C(\lambda + \sigma) - \frac{(l+2j)^2}{2} ) \phi_{\lambda + \sigma;p+2i,l+2j} - (C(2 \omega_1) - 2) \phi_{\lambda + \sigma;p+2i,l+2j} - (C(\lambda) - \frac{l^2}{2}) \phi_{\lambda + \sigma;p+2i,l+2j}\right)
\\ & = \frac{1}{2} \sum_{i,j = \pm 1} \sum_{\sigma} C^i(\sigma,\lambda,p) C^j(\sigma,\lambda,l)  (C(\lambda + \sigma) - \frac{(l+2j)^2}{2} - C(2 \omega_1) + 2 - C(\lambda) + \frac{l^2}{2})) \phi_{\lambda + \sigma;p+2i,l+2j}
\\ & = \frac{1}{2} \sum_{i,j = \pm 1} \sum_{\sigma} C^i(\sigma,\lambda,p) C^j(\sigma,\lambda,l) (C(\lambda + \sigma) - C(2 \omega_1) - C(\lambda) - 2j l) \phi_{\lambda + \sigma;p+2i,l+2j}.
\end{flalign*}
%All together, if $\nu_2 = 0$ this would make
%\begin{flalign*}
%& \pi_{\nu}(E) \phi_{\lambda;p,l} = \sum_{i,j = \pm 1} \sum_{\sigma} C^i(\sigma,\lambda,p) C^j(\sigma, \lambda,l) \left( \frac{\nu_1}{2} + (\frac{1}{2} - 1)j l + \frac{1}{2} (C(\lambda + \sigma) - C(2 \omega_1) - C(\lambda)) \right) \phi_{\lambda + \sigma; p+2i,l+2j}
%\\ & = \sum_{i,j = \pm 1} \sum_{\sigma} C^i(\sigma,\lambda,p) C^j(\sigma, \lambda,l) \left( \frac{\nu_1}{2} - \frac{1}{2} j l + \frac{1}{2} (C(\lambda + \sigma) - C(2 \omega_1) - C(\lambda)) \right) \phi_{\lambda + \sigma; p+2i,l+2j}
%\end{flalign*}
This proves the theorem.
\end{proof}

An inspection of the proof of Theorem \ref{actiondiffop} shows that if we replace $\phi_{\lambda;p,l}$ by $\phi^{\pm}_{\lambda;p,l} = \phi_{\lambda;p,l} \pm w_0 \phi_{\lambda;p,l}$ the same formula holds, where $l \geq 0$. We can get $l = -2$ from $\phi_{\lambda;p,0}^+$.
%We first note that according to the Euclidean inner product, the positive roots are $\{e_i \pm e_j \ | \ i < j \} \cup \{e_i \}$ for $SO(2n+1)$ and $\{e_i \pm e_j \ | \ i < j \}$ for $SO(2n)$. Then $\omega_1 = e_1$ and we see that for our orthonormal base for the Killing form we had $\frac{\sqrt{2}}{2} \begin{pmatrix} \begin{pmatrix} 0 & i \\ -i & 0 \end{pmatrix} & 0 \\ 0 & 0 \end{pmatrix}$, for $(X,Y) \mapsto \Tr(XY)$, call these the $H_{\alpha}$. Thus we see that $\langle \omega_1, \omega_1 \rangle = \sum_{\alpha} \omega_1(H_{\alpha})^2 = \frac{1}{2}$, and we conclude that $\langle -,-\rangle = \frac{1}{2} \langle -,- \rangle_{\mathrm{Euclidean}} = \frac{1}{2} \langle -,- \rangle_{E}$.
Let $\rho_{\so(n+2,\C)}$ be the half-sum of the positive roots. Then from \cite[Appendix C]{Kna02} we see that for the Casimir element and $\lambda = (\lambda_1 + \lambda_2) \varepsilon_1 + \lambda_2 \varepsilon_2$
\begin{flalign*}
& C(\lambda) = \langle \lambda + 2 \rho_{\so(n+2,\C)}, \lambda \rangle = \frac{(\lambda_1 + \lambda_2 + n)(\lambda_1 + \lambda_2) + (\lambda_2 + n - 2)\lambda_2}{2}
\\ & = \frac{(\lambda_1 + \lambda_2 + n) \lambda_1 + (\lambda_1 + 2 \lambda_2 + 2n -2) \lambda_2}{2}.
\end{flalign*}
We let
$$ B(\nu,\lambda,\sigma,l,j) = \nu_1  - j l + (C(\lambda + \sigma) - C(2 \omega_1) - C(\lambda))$$
for $j = \pm 1$. As a convention, $B(\nu,\lambda,\sigma,l,0) = \frac{2 \nu_2 (n+2)}{n}$. Then from Theorem \ref{actiondiffop}
\begin{equation}
\label{actiondiffopeq}
\pi_{\delta_1,\delta_2,\nu}(E) \phi_{\lambda;p,l}^{\pm} = \frac{1}{2} \sum_{i = \pm 1,j = 0,\pm 1} \sum_{\sigma}  B(\nu,\lambda,\sigma,l,j) C^i(\sigma,\lambda,p) C^j(\sigma,\lambda,l) \phi_{\lambda+\sigma;p+2i,l+2j}^{\pm}.
\end{equation}

We make a list of the values for $B(\nu,\lambda,\sigma,l,j)$, obtained by direct calculation.

\begin{lemma}
\label{valconstB}
We have the following values for $B(\nu,\lambda,\sigma,l,j)$, $j = \pm 1$:
\begin{flalign*}
& B(\nu,\lambda,(0,0),l,j) = \nu_1 - j l - n - 2,
\\ & B(\nu,\lambda,(0,1),l,j) = \nu_1 - j l + (\lambda_1 + 2 \lambda_2 - 2),
\\ & B(\nu,\lambda,(2,-1),l,j) = \nu_1 - j l + (\lambda_1 - n),
\\ & B(\nu,\lambda,(2,0),l,j) = \nu_1 - j l + (2 \lambda_1 + 2 \lambda_2),
\\ & B(\nu,\lambda,(2,-2),l,j) = \nu_1 - j l + (- 2 \lambda_2 - 2 n + 2),
\\ & B(\nu,\lambda,(0,-1),l,j) = \nu_1 - j l + (- \lambda_1 - 2 \lambda_2 - 2 n),
\\ & B(\nu,\lambda,(-2,1),l,j) = \nu_1 - j l + (- \lambda_1 - n - 2),
\\ & B(\nu,\lambda,(-2,0),l,j) = \nu_1 - j l + (- 2 \lambda_1 - 2 \lambda_2 - 2 n ),
\\ & B(\nu,\lambda,(-2,2),l,j) = \nu_1 - j l + (2 \lambda_2 - 2).
\end{flalign*}
\end{lemma}

%We also know that $\rho = (1+n) E^*$, so $\widetilde{\nu} = 2 \rho - \nu = (2 + 2n - \nu_1) E^*$. Note that $\nu \in \Z$ iff $\widetilde{\nu} \in \Z$ and $\nu \in 2 \Z$ iff $\widetilde{\nu} \in 2 \Z$.

\section{Points of reducibility}
\label{SecRed}

In this section we will determine the points $\nu = (\nu_1,\nu_2)$ of reducibility and at which points the induced representation is irreducible. Recall $\rho_A$ from Section \ref{sec-heis-prel}. We write $\widetilde{\nu} = 2 \rho_A - \nu = (2n+2 - \nu_1, - \nu_2)$.

\subsection{Knapp-Stein operator and reducibility}

%We immediately see from Lemma \ref{valconstB} that this implies that for $\nu \notin \Z$ this holds.

We recall an important family of intertwining operators for the analysis of the points at which the induced representation is irreducible. By the theory of Knapp and Stein \cite[Chapters VII, VIII]{Kna86} there exists a meromorphic family of $G$-intertwining operators
$$A(\delta_1,\delta_2,\nu): I(\delta_1,\delta_2,\nu) \rightarrow I(\delta_2,\delta_1,\widetilde{\nu}).$$
As $I(\delta_1,\delta_2,\nu) \cong L^2(K/L,\chi_{\delta_1, \delta_2}) \cong \bigoplus V_{\lambda} \otimes \overline{V_{\lambda}^{\chi_{\delta_1,\delta_2}}}$, by Schur's Lemma the restriction of $A(\delta_1,\delta_2,\nu)$ to $\li{V_{\lambda}^{\chi_{\delta_1,\delta_2}}}$ is given by $id_{\lambda} \otimes A(\lambda_1,\lambda_2)$, where $A(\lambda_1,\lambda_2)$ is a matrix
$$A(\lambda_1,\lambda_2): \li{V_{(\lambda_1, \lambda_2)}^{\chi_{\delta_1, \delta_2}}} \rightarrow \li{V_{(\lambda_1, \lambda_2)}^{\chi_{\delta_2, \delta_1}}}.$$
Note that we get information and a basis by Equation \ref{decompL2}.

We can now divide $\pi_{\delta_1,\delta_2,\nu}(E)$ up in the $K$-types as they appear in Theorem \ref{actiondiffop}, so restricting the image to the $K$-type $\lambda + \sigma$ we have an operator
$$\li{T_{(\lambda_1, \lambda_2)}^{(\lambda_1 + \sigma_1, \lambda_2 + \sigma_2)}(\delta_1,\delta_2,\nu)}: \li{V_{(\lambda_1, \lambda_2)}^{\chi_{\delta_1, \delta_2}}} \rightarrow \li{V_{(\lambda_1, \lambda_2)}^{\chi_{\delta_1, \delta_2}}}.$$
Note that we will instead consider the operator $T_{(\lambda_1, \lambda_2)}^{(\lambda_1 + \sigma_1, \lambda_2 + \sigma_2)}(\delta_1,\delta_2,\nu)$ for ease of notation, so we do not have to consider the complex conjugate. As we will be looking at reduction points and zeros, it will suffice to consider this operator. Note that by Equation \ref{actiondiffopeq}, writing $v_{\lambda}^{l, \pm} = v^l_{\lambda} + (-1)^{\delta_2} \cdot w_0  \cdot v^{l}_{\lambda}$ where $\pm$ depends on the sign of $(-1)^{\delta_2}$, we obtain
$$ T_{(\lambda_1,\lambda_2)}^{(\lambda_1+\sigma_1,\lambda_2+\sigma_2)}(\delta_1,\delta_2,\nu) v^{l,\pm}_{\lambda} =  \sum_{j =0,\pm 1} \sum_{\sigma} B(\nu,\lambda,\sigma,l,j) C^j(\sigma,\lambda,l) v^{l+2j,\pm}_{\lambda+\sigma}.$$
Furthermore, for $v \in V_{(\lambda_1, \lambda_2)}^{\chi_{\delta_1, \delta_2}}$
\begin{equation}
\label{KSintertwiner}
\sum_{\sigma} T_{(\lambda_1,\lambda_2)}^{(\lambda_1 + \sigma_1, \lambda_2 + \sigma_2)}(\delta_2,\delta_1,\widetilde{\nu}) A(\lambda_1, \lambda_2) v = \sum_{\sigma} A(\lambda_1 + \sigma_1, \lambda_2 + \sigma_2) T_{(\lambda_1,\lambda_2)}^{(\lambda_1 + \sigma_1, \lambda_2 + \sigma_2)}(\delta_1,\delta_2,\nu) v.
\end{equation}

We first find some points of reducibility for $\chi_{0,0}$ and $\chi_{1,1}$. The proofs of reducibility use the kernel of the Knapp-Stein operator as in \cite{FHWZ26}. $A(\delta_1,\delta_2,\nu)$ can always be normalized to be non-zero and its kernel, if it exists, is then a non-trivial subrepresentation. In other words, it is in a higher level of the Jantzen filtration \cite[Definition 3.7]{Vog84}.

\begin{proposition}
\label{firstredots}
For $I(0,0,\nu)$ and $I(1,1,\nu)$ we have reduction points at $(\nu_1, \nu_2)$ where $\nu_1 \in n+2 \Z$ and $\nu_2 \in \C$.
\end{proposition}

\begin{proof}
We will use the fact that we can normalize the Knapp-Stein operator such that it is non-zero. To prove it has a non-zero kernel we use Equation (\ref{KSintertwiner}). If we can find some $T_{(\lambda_1,\lambda_2)}^{(\lambda_1 + \sigma_1, \lambda_2 + \sigma_2)}(\delta,\delta, \nu)$ in the left-hand side of Equation (\ref{KSintertwiner}) which has a non-trivial kernel such that $T_{(\lambda_1,\lambda_2)}^{(\lambda_1 + \sigma_1, \lambda_2 + \sigma_2)}(\delta,\delta,\widetilde{\nu})$ has a trivial kernel, we have proved some of the vectors vanish of a higher degree than other vectors and thus the Knapp-Stein operator must have a kernel.

We study $T_{(2m,2)}^{(2m+2,1)}(\delta,\delta,\nu)$ and note that
$$T_{(2m,2)}^{(2m+2,1)}(0,0,\nu) v_{(2m,2)}^{0,+} =  2 \left( \nu_1 + (2m - n) \right) C^1((2,-1),(2m,2),0) v_{(2m+2,1)}^{2,+},$$
and
$$T_{(2m,2)}^{(2m+2,1)}(1,1,\nu) v_{(2m,1)}^{0,-} =  2 \left( \nu_1 + (2m - n) \right) C^1((2,-1),(2m,1),0) v_{(2m+2,0)}^{2,-}.$$
Observe that $0 \neq v_{(2m,2)}^{0,+} \in \li{V^{\chi_{0,0}}_{(2m,2)}}$ and $0 \neq v_{(2m,1)}^{0,-} \in \li{V^{\chi_{1,1}}_{(2m,1)}}$.
Thus for $\nu_1 \in n + 2 \Z_{\leq 0}$ we have a non-zero kernel. However, $T_{(2m,2)}^{(2m+2,1)}(\delta,\delta, \widetilde{\nu})$ has coefficients $2n + 2 - \nu_1 - jl + (2m - n) = n + 2 - \nu_1 + (2m - jl) > 0$, so this does not have a kernel. $\nu_1 \in n + 2 \Z$ is obtained by considering $\widetilde{\nu}$.
%Now we study $T_{(2m,0)}^{(2m+2,0)}(\nu)$ and see that $T_{(2m,0)}^{(2m+2,0)}(\nu) v_{(2m,2)}^{0,+} =  \left( \nu_1 + 4 m \right) v_{(2m+2,1)}^{2,+}$. Thus for $\nu_1 \in 4 \Z_{\leq 0}$ we get a kernel. However, $T_{(2m,0)}^{(2m+2,0)}(\widetilde{\nu})$ has coefficients $2n + 2 - \nu_1 - jl + 4m > 0$, so this does not have a kernel.
%We see know that $\rho = (1+n)E^*$. The reduction points are $n + 2 \Z \cup 4 \Z_{\leq 0} \cup 2 + 2n + 4 \Z_{\geq 0}$. For the parameter $\nu_1 = n+1$ and $\nu_2 \in i \R$ we have a unitary representation, and we have a Weyl group element $w = ( \begin{pmatrix} 0 & 1 \\ 1 & 0 \end{pmatrix},-1,I_{n-1})$ so that $w^2 = 1$, $w \cdot id = id$ and $w \cdot \nu = - \overline{\nu}$.
\end{proof}

%We recall that for $\epsilon_2$
%$$L^2(SO(n+2)/L, \epsilon_2) \cong \sum_{\lambda_1 \in 2 \Z_{\geq 0}} V_{\lambda_1, \lambda_2} \otimes V_{\lambda_1,\lambda_2}^{L, \epsilon_2}$$
%where a base for $V_{\lambda_1,\lambda_2}^{L, \epsilon_2}$ is $\{ v_{\lambda_1, \lambda_2}^{l-} = v_{\lambda_1, \lambda_2}^l - w_0 v_{\lambda_1, \lambda_2}^l \}_{0 \leq l \leq \lambda_2, l \equiv 0 \mod 2}$. Then we have the following lemma.

Now we find some other points of reducibility.

\begin{proposition}
\label{redptsnu1nu2}
If $\nu_1 \pm \frac{n+2}{n} \nu_2 \in 4 \Z_{\leq 0} \cup (2n+2) + 4 \Z_{\geq 0}$ we have a reduction point for $I(0,0,\nu)$. If $\nu_1 \pm \frac{n+2}{n} \nu_2 \in (-2 + 4 \Z_{\leq 0}) \cup (2n+4) + 4 \Z_{\geq 0}$ we have a reduction point for $I(1,1,\nu)$.
%for $\nu$ or $\widetilde{\nu}$ we have a relation $\nu_2 = \pm \frac{n}{n+2} (\nu_1 + 4k)$ for some $k \in \Z_{\geq 0}$, we have a reduction point for $I(0,0,\nu)$. If for $\nu$ or $\widetilde{\nu}$ we have a relation $\nu_2 = \pm \frac{n}{n+2} (\nu_1 + 4k + 2)$ for $k \in \Z_{\geq 0}$, that is a reduction point for $I(1,1,\nu)$.
\end{proposition}

\begin{proof}
We prove the fact for $I(0,0,\nu)$; the case for $I(1,1,\nu)$ is similar. We prove that if for $\nu$ or $\widetilde{\nu}$
$$\nu_2 = \pm \frac{n}{n+2} (\nu_1 + 4k)$$
for some $k \in \Z_{\geq 0}$, we have a reduction point for $I(0,0,\nu)$, which is equivalent to the formulation above. Let $\lambda_2$ be even. We apply $2 \pi_{0,0}( E)$ (the $2$ is added for more convenient equation) two times and check the kernel of $T_{(2,\lambda_2)}^{(0,\lambda_2+2)}(\nu) T_{(0,\lambda_2)}^{(2,\lambda_2)}(\nu)$. We get
\begin{flalign*}
& 4 T_{(2,\lambda_2)}^{(0,\lambda_2+2)}(0,0,\nu) T_{(0, \lambda_2)}^{(2,\lambda_2)}(0,0,\nu) v_{(0,\lambda_2)}^{0,+}
\\ & = ( (\frac{2 (n+2)\nu_2}{n})^2 C^0((-2,2),(2,\lambda_2),0) C^0((2,0),(0,\lambda_2),0) +
\\ & 2 (\nu_1 + 2 \lambda_2)^2 C^{-1}((-2,2),(2,\lambda_2),2) C^1((2,0),(0,\lambda_2),0) ) v_{(0, \lambda_2+2)}^{0,+}.
\end{flalign*}
Now we furthermore see that, if $w_2$ is the Weyl group element sending $\varepsilon_1 \mapsto \varepsilon_2$, as in Section \ref{cfctncalc}, we can identify $(-2,2)$ with $2 w_2 \cdot \omega_1$. By Corollary \ref{multformsph} and the explicit formulas in Proposition \ref{explicitcfctn} it follows that there is a constant $C \in \R_{>0}$ such that
%$$|C^0((-2,2),(2,\lambda_2),0) C^0((2,0),(0,\lambda_2),0)| = \sqrt{A_{w_2, (2,\lambda_2)_c +\rho+e_2,1}^- A_{w_2, (2,\lambda_2)_c+\rho,0} A_{1, (0,\lambda_2)_c+\rho+e_1,1}^- A_{1, (0,\lambda_2)_c+\rho,0}},$$
%and
%$$|C^{-1}((-2,2),(2,\lambda_2),2) C^1((2,0),(0,\lambda_2),0)| = \sqrt{A_{w_2, (2,\lambda_2)_c+\rho+e_2,1}^- A_{w_2, (2,\lambda_2)_c +\rho,2}^- A_{1, (0,\lambda_2)_c +\rho+e_1,1} A_{1, (0,\lambda_2)_c +\rho,0}}.$$
$$|C^0((-2,2),(2,\lambda_2),0) C^0((2,0),(0,\lambda_2),0)| =  (4 + 2 \lambda_2 + a) (2 + 2 \lambda_2 + a) C,$$
and
$$|C^{-1}((-2,2),(2,\lambda_2),2) C^1((2,0),(0,\lambda_2),0)| = 2 (4 + 2 \lambda_2 + a)(2+2 \lambda_2 + a) C.$$
%By Proposition \ref{explicitcfctn} there is a constant $C$ such that
%\begin{flalign*}
%& A_{w_2, (2,\lambda_2)_c +\rho+e_2,1}^- A_{w_2, (2,\lambda_2)_c +\rho,0} A_{1, (0,\lambda_2)_c +\rho+e_1,1}^- A_{1, (0,\lambda_2)_c +\rho,0}
%\\ & = 16 (4 + 2 \lambda_2 + a)^2(2 + 2 \lambda_2 + a)^2 C,
%\end{flalign*}
%and
%\begin{flalign*}
%& A_{w_2, (2,\lambda_2)_c +\rho+e_2,1}^- A_{w_2, (2,\lambda_2)_c +\rho,2}^- A_{1, (0,\lambda_2)_c +\rho+e_1,1} A_{1, (0,\lambda_2)_c +\rho,0}
%\\ & = 64 (4 + 2 \lambda_2 + a)^2(2+2 \lambda_2 + a)^2 C.
%\end{flalign*}
Lemma \ref{oppsigns} now implies that
$$ T_{(2,\lambda_2)}^{(0,\lambda_2+2)}(0,0,\nu) T_{(0, \lambda_2)}^{(2,\lambda_2)}(0,0,\nu) v_{(0,\lambda_2)}^{0,+} = \left( (\frac{(n+2) \nu_2}{n} )^2 - (\nu_1 + 2 \lambda_2)^2 \right) C' v_{(0,\lambda_2+2)}^{0,+},$$
where $C' \in \C^*$. This coefficient is zero if $\frac{(n+2)\nu_2}{n} = \pm (\nu_1 + 2 \lambda_2)$.

By a similar argument as in Proposition \ref{firstredots} and the equation $A(\nu) \pi_{\nu}(E)^2 = \pi_{\widetilde{\nu}}(E)^2 A(\nu)$ we get a point of reducibility if $\frac{(n+2)\nu_2}{n} = \pm (\nu_1 + 2 \lambda_2)$ and this is not the case for $\widetilde{\nu} = (2n+2 - \nu_1, - \nu_2)$, i.e.
$$- \frac{(n+2)\nu_2}{n} \neq \pm (2n+2 - \nu_1 + 2 \lambda_2).$$
The only time $\frac{(n+2)\nu_2}{n} = \pm (\nu_1 + 2 \lambda_2)$ and $- \frac{(n+2)\nu_2}{n} = \pm (2n+2 - \nu_1 + 2 \lambda_2)$ holds is when $\nu_1 = n+1$, so we treat this case separately.

For $\nu_1 = n+1$, we use a result by Howe and Lee \cite{HoweLee99}. Let $H_1 = E + F$, $H_2 = E - \frac{n+2}{n} F$, $M' = SL(n+1,\R)^{\pm}$, $\af' = \R H_1$ and $\af'' = \R H_2$. By the double induction formula \cite[Chapter VII]{Kna86} we see that
$$ Ind_{MAN}^{SL(n+2,\R)}( \sigma \otimes e^{\nu_1 E^* + \nu_2 F^*} \otimes 1) \cong Ind_{M'A'N'}^{SL(n+2,\R)} \left( Ind_{M A'' N''}^{M'}(\sigma \otimes e^{(\nu_1 - \frac{n+2}{n} \nu_2)H_2^*} \otimes 1 ) \otimes e^{(\nu_1 + \nu_2)H_1^*} \otimes 1 \right).$$
For $\nu_2 = \frac{n}{n+2} (\nu_1 + 4k)$, we get an induced representation from $Ind_{M A'' N''}^{M'}(\sigma \otimes e^{-4k} \otimes 1 )$. By the results in \cite{HoweLee99} we get that this induced representation is reducible for $-4k$, where $k \in \Z_{\geq 0}$. For reducibility at $\nu_2 = - \frac{n}{n+2} (n+1 + 4k)$ we use that this is exactly $\widetilde{\nu}$ for $\nu_2 = \frac{n}{n+2} (n+1 + 4k)$.
\end{proof}

We prove a similar statement for the other characters.

\begin{proposition}
\label{ptsredothchar}
When we have a relation $\nu_1 - \frac{n+2}{n} \nu_2 \in 4 \Z_{\leq 0} \cup (2n+2) + 4 \Z_{\geq 0}$ or $\nu_1 + \frac{n+2}{n} \nu_2 \in (-2 + 4 \Z_{\leq 0}) \cup (2n+4) + 4 \Z_{\geq 0}$ we have a reduction point for $I(1,0,\nu)$. When we have a relation $\nu_1 + \frac{n+2}{n} \nu_2 \in 4 \Z_{\leq 0} \cup (2n+2) + 4 \Z_{\geq 0}$ or $\nu_1 - \frac{n+2}{n} \nu_2 \in (-2 + 4 \Z_{\leq 0}) \cup (2n+4) + 4 \Z_{\geq 0}$ we have a reduction point for $I(0,1,\nu)$.
%When for $\nu$ or $\widetilde{\nu}$ we have a relation $\frac{n+2}{n} \nu_2 = (-1)^{k} (\nu_1 + 2 k)$ for some $k \in \Z_{\geq 0}$ we have a reduction point for $I(1,0,\nu)$. When for $\nu$ or $\widetilde{\nu}$ we have a relation $\frac{n+2}{n} \nu_2 = (-1)^{k+1} (\nu_1 + 2 k)$ for some $k \in \Z_{\geq 0}$ we have a reduction point for $I(0,1,\nu2)$.
\end{proposition}

\begin{proof}
Using Lemma \ref{oppsgnsothchar} of the appendix the result is derived in the same way as for Lemma \ref{redptsnu1nu2}. Note that the Knapp-Stein operator intertwines the induced representations for $\chi_{(1,0)}$ and $\chi_{(0,1)}$.
\end{proof}

\subsection{Irreducible points}

We will show the points at which the induced representation is irreducible, found above, are all of them. We will first need a lemma, which was also used to analyze reducibility as in \cite[Lemma 4.1]{FHWZ26}.

\begin{lemma}
\label{irredcrit}
The representation $I(\nu)$ is irreducible if and only if $I(\nu,\delta_1,\delta_2)$ and $I(\widetilde{\nu},\delta_1,\delta_2)$ are generated by the same multiplicity-one $K$-type.
\end{lemma}

Now we show that at the points not appearing Propositions \ref{redptsnu1nu2} and \ref{ptsredothchar}, the induced representation is irreducible. Our proof will rely on asymptotic values of the $c$-function and some results on linear independence. We recall the constants $C^i(\sigma,\lambda,l)$ from Corollary \ref{multformsph}. We will first need the following lemma. We let $w_1$ again be a Weyl group element sending $\varepsilon_1$ to $-\varepsilon_2$, as in Section \ref{cfctncalc}.

\begin{lemma}
\label{linindep}
Let $\lambda = (\lambda_1,\lambda_2)$ and $\sigma = (2,0) = 2 \omega_1$ or $\sigma = (2,-2) = 2 w_1 \cdot \omega_1$. Then the limit $\lim_{\lambda_2 \rightarrow \infty} |C^i(\sigma,\lambda,l)|$ exists and can be made explicit, and for $l > 0$ and $\lambda_2$ big enough the vectors 
$$\sum_{j= \pm 1} B(\nu,(\lambda_1, \lambda_2),(2,0),l,j) C^j((2,0), (\lambda_1, \lambda_2),l) v^{l+2 j, +}_{\lambda}$$
and 
$$\sum_{j= \pm 1} B(\nu, (\lambda_1, \lambda_2 + 2),(2,-2),l,j) C^j((2,-2), (\lambda_1, \lambda_2 + 2),l) v^{l+2 j, +}_{\lambda}$$
are linearly independent.
\end{lemma}

\begin{proof}
From the explicit formulas for $|C^1(2 w \cdot \omega_1,\lambda,l)|^2$ in Corollary \ref{multformsph} and $A_{1,\lambda_c + \rho,l}$ in Proposition \ref{explicitcfctn} we see that
%$$|C^1((2,0),(\lambda_1, \lambda_2),l)|^2 = A_{1,\lambda + \rho,l} A_{1,\lambda + \rho + e_1,l+1},$$
%$$|C^{-1}((2,0),(\lambda_1, \lambda_2),l)|^2 = A_{1,\lambda + \rho,-l} A_{1,\lambda + \rho + e_1, - l + 1},$$
%$$ |C^1((2,-2),(\lambda_1, \lambda_2),l)|^2 = A_{w,\lambda + \rho,l} A_{w,\lambda + \rho - e_2,l+1},$$
%and
%$$ |C^{-1}((2,-2),(\lambda_1, \lambda_2),l)|^2 = A_{w,\lambda + \rho,-l} A_{w,\lambda + \rho - e_2,-l+1}.$$
%We have the explicit formulas
%$ A_{1,\lambda + \rho,l} = \frac{a + 1 + \lambda_1 + \lambda_2}{4 \lambda_1 + 4 \lambda_2 + 2a + 4} \cdot \frac{2 + \lambda_1 + 2 \lambda_2 + a + l}{\lambda_1 + 2 \lambda_2 + a + 1} \cdot \frac{\lambda_1 + 2 + l}{\lambda_1 + 1}$$
%and
%$$A_{w,\lambda + \rho,l} = \frac{\lambda_2}{4 \lambda_2 + 2a } \cdot \frac{ \lambda_1 + 2 \lambda_2 + a - l}{ \lambda_1 + 2 \lambda_2 + a + 1} \cdot \frac{2 + \lambda_1 - l}{\lambda_1 + 1}.$$
%It follows that
\begin{flalign*}
& \lim_{\lambda_2 \rightarrow \infty} |C^1((2,0),(\lambda_1, \lambda_2),l)|^2 = \frac{\lambda_1 + 2 + l}{4(\lambda_1 +1)} \frac{\lambda_1 + 4 + l}{4(\lambda_1 + 2)}
\\ & \lim_{\lambda_2 \rightarrow \infty} |C^{-1}((2,0),(\lambda_1, \lambda_2),l)|^2 = \frac{\lambda_1 + 2 - l}{4(\lambda_1 +1)} \frac{\lambda_1 + 4 - l}{4(\lambda_1 + 2)}
\\ & \lim_{\lambda_2 \rightarrow \infty} |C^{0}((2,0),(\lambda_1, \lambda_2),l)|^2 = \frac{ (\lambda_1 + 2 - l)^2}{16(\lambda_1 +1)(\lambda_1 + 2)}
\\ & \lim_{\lambda_2 \rightarrow \infty} |C^1((2,-2),(\lambda_1, \lambda_2),l)|^2 = \frac{\lambda_1 + 2 - l}{4 ( \lambda_1 + 1)} \frac{\lambda_1 + 2 - l}{4 ( \lambda_1 + 2)}
\\ & \lim_{\lambda_2 \rightarrow \infty} |C^{-1}((2,-2),(\lambda_1, \lambda_2),l)|^2 = \frac{\lambda_1 + 2 + l}{4 ( \lambda_1 + 1)} \frac{\lambda_1 + 2 + l}{4 ( \lambda_1 + 2)}
\\ & \lim_{\lambda_2 \rightarrow \infty} |C^{0}((2,-2),(\lambda_1, \lambda_2),l)|^2 = \frac{\lambda_1 + 2 + l}{4 ( \lambda_1 + 1)} \frac{\lambda_1 + 4 + l}{4 ( \lambda_1 + 2)}
\end{flalign*}

We obtain linear independence from the limit formulas and the fact that $B(\nu,(\lambda_1, \lambda_2),(2,0),l,j)$ and $B(\nu,(\lambda_1, \lambda_2),(2,-2),l,j)$ both grow like $\lambda_2$, $-\lambda_2$ respectively independently of $j$.
\end{proof}

Now we prove irreducibility. Note that we prove it for $n \geq 3$; we treat $n = 2$ in Remark \ref{nis2case}.

\begin{proposition}
\label{repirrednu1nu2}
When $\nu_1 \notin n + 2 \Z$ and $\nu_1 \pm \frac{n+2}{n} \nu_2 \notin 4 \Z_{\leq 0} \cup (2n+2) + 4 \Z_{\geq 0}$ then $V_{(0,0)}$ generates $I(0,0,\nu)$ and $I(0,0,\widetilde{\nu})$ and we do not have a reduction point.

When $\nu_1 \notin n + 2 \Z$ and $\nu_1 \pm \frac{n+2}{n} \nu_2 \notin (-2 + 4 \Z_{\leq 0}) \cup (2n+4) + 4 \Z_{\geq 0}$ then $V_{(0,0)}$ generates $I(1,1,\nu)$ and $I(1,1,\widetilde{\nu})$ and we do not have a reduction point.
\end{proposition}

\begin{proof}
We prove the proposition only for $I(0,0,\nu)$ and $I(0,0,\widetilde{\nu})$; the other case follows similarly. We prove $V_{(0,0)}$ generates $I(0,0,\nu)$ and $I(0,0,\widetilde{\nu})$ and by Lemma \ref{irredcrit} there will be no reduction point. In fact, first we will prove that every $K$-type $V_{(0,\lambda_2)}$ and $V_{(2,\lambda_2)}$ for each $\lambda_2$ is generated by $V_{(0,0)}$. Then we prove that for $\lambda_2$ big enough, we generate every $K$-type $V_{(\lambda_1,\lambda_2)}$ for each $\lambda_1$. Thirdly, we prove that for $\lambda_1$ big enough, we generate $V_{(\lambda_1,\lambda_2)}$ for \emph{every} $\lambda_2$. Then we argue the quotient by this module must be finite-dimensional, which cannot happen.

By the proof of Proposition \ref{redptsnu1nu2} we see that $V_{(0,0)}$ generates every $V_{(0,\lambda_2)}$ for $\lambda_2$ even; observe all these $K$-types have multiplicity one. Now we want to generate the $K$-types $V_{(\lambda_1,\lambda_2)}$ where $\lambda_1 = 2$. Using Lemma \ref{valconstB} we see that using $T_{\lambda}^{\lambda + \sigma}(\delta_1,\delta_2,\nu)$ and $\sigma = (2,-1)$ we generate $v^{+,2}_{(2, \lambda_2-1)}$ for every $\lambda_2$ even, as $\nu_1 \notin n + 2 \Z$. Setting $\sigma = (2,0)$, from Lemma \ref{valconstB} and Lemma \ref{linindep} we generate from $v^{+,0}_{(0,\lambda_2)}$ a linear combination of $v^{+,0}_{(2,\lambda_2)}$ and $v^{+,2}_{(2,\lambda_2)}$, and the coefficient of $v^{+,2}_{(2,\lambda_2)}$ is $O(\lambda_2)$ as $\lambda_2$ big, but $O(1)$ for $v^{+,0}_{(2,\lambda_2)}$. Similarly, setting $\sigma = (0,-1)$ we see from Proposition \ref{explicitcfctn} that we generate from $v_{(2,\lambda_2+1)}^{+,2}$ a linear combination of $v^{+,0}_{(2,\lambda_2)}$ with coefficient $O(\lambda_2)$ and $v^{+,2}_{(2,\lambda_2)}$ with coefficient $O(1)$. As a consequence, for $\lambda_2$ big enough we generate both $v^{+,0}_{(2,\lambda_2)}$ and $v^{+,2}_{(2,\lambda_2)}$. %Then we get $(2, \lambda_2)$ $\lambda_2$ even and big enough.

Now we want to generate the $K$-types $V_{(\lambda_1,\lambda_2)}$ for any $\lambda_1$ when $\lambda_2$ is big. Using $T_{\lambda}^{\lambda + \sigma}(\delta_1,\delta_2,\nu)$ and $\sigma = (2,0)$ we obtain $\dim \li{ V_{(\lambda_1,\lambda_2)}^{\chi_{\delta_1,\delta_2}}}$ linearly independent vectors. From Lemma \ref{linindep} and the fact that $B(\nu,\lambda,\sigma,l,j)$ grows like $\lambda_2$ for $l=0,2$, $\sigma \in \{(2,-2),(2,0) \}$ and $j = \pm 1$, while it grows like $O(1)$ for $j = 0$, the vectors $T_{\lambda}^{\lambda + (2,0)}(\delta_1,\delta_2,\nu)(v^{0,+}_{(\lambda_1,\lambda_2)})$, $T_{\lambda}^{\lambda + (2,0)}(\delta_1,\delta_2,\nu)(v^{2,+}_{(\lambda_1,\lambda_2)})$ and $T_{\lambda}^{\lambda + (2,-2)}(\delta_1,\delta_2,\nu)(v^{2,+}_{(\lambda_1,\lambda_2)})$ are linearly independent and we obtain all the $K$-types $V_{(\lambda_1+2,\lambda_2)}$. Together with the map $T^{\lambda + \sigma}_{\lambda}$ for $\sigma = (2,0)$, we generate every $K$-type, as $B(\nu,\lambda,\sigma,l,j)$ will be enough to generate the remaining vectors from 

%Let $\lambda_2$ be big, then by Lemma \ref{linindep} we can see that independent of our $\lambda_1, \nu_1$ we can produce two linearly independent elements containing $v^{+,0}$ and $v^{+,4}$ (or any other two elements for that matter). Note that if $\lambda_2$ big enough then the term coming from $F$, or projecting on themselves, do not matter (if $l$ small) as $(2,0)$ and $(2,-2)$ both have coefficients that scale like $\lambda_2$ and these do not appear for $F$.

Now we consider the case when $\lambda_1$ is big and $\lambda_2$ is fixed; note that if $\lambda_2$ is big enough, we generate every $K$-type $V_{(\lambda_1, \lambda_2)}$ with full multiplicity, independently of $\lambda_1$. From Proposition \ref{explicitcfctn} and Lemma \ref{valconstB} we obtain
$$T_{(\lambda_1, \lambda_2+1)}^{\lambda}(\delta_1, \delta_2, \nu)( v^{\lambda_1,+}) = O( \frac{1}{\lambda_1}) v^{\lambda_1,+} + O(1) v^{\lambda_1 - 2,+},$$
$$T_{(\lambda_1-2,\lambda_2+2)}^{\lambda}(\delta_1, \delta_2, \nu)(v^{\lambda_1 - 2,+}) = O(\frac{1}{\lambda_1}) v^{\lambda_1,+} + O(\frac{1}{\lambda_1}) v^{\lambda_1 - 2,+} + O(\lambda_1) v^{\lambda_1-4,+},$$
and
$$T_{(\lambda_1-2,\lambda_2+1)}^{\lambda}(\delta_1, \delta_2, \nu)(v^{\lambda_1-2,+}) = O(\frac{1}{\lambda_1}) v^{\lambda_1,+} + O(1) v^{\lambda_1-2,+} + O(1) v^{\lambda_1-4,+}.$$
When $\lambda_1$ is big, these are linearly independent vectors. Now we can generate the $K$-type $V_{(\lambda_1, \lambda_2)}$ with full multiplicity using these vectors and the $T_{\lambda}^{\lambda+\sigma}$, and as $\lambda_1$ is big we generate $(\lambda_1,\lambda_2)$ for \emph{every} $\lambda_2$.

%We can go to $(\lambda_1,\lambda_2)$ from $(\lambda_1, \lambda_2+1)$ by $\sigma = (0,-1)$ and seeing the coefficients before $\sigma$ and from the $c$-function from Lemma we see that $v^{\lambda_1,+}$ gets sent to $O( \frac{1}{\lambda_1}) v^{\lambda_1,+} + O(1) v^{\lambda_1 - 2,+}$,
%and if we apply $\sigma = (0,0)$ we get $O( \frac{1}{\lambda_1}) v^{\lambda_1,+} + O(1) v^{\lambda_1 - 2,+} + O(\frac{1}{\sqrt{\lambda_1}})v^{\lambda_1 - 4,+}$.
%We have that $(\lambda_1 - 2, \lambda_2 + 2)$ using $(2,-2)$ sends $v^{\lambda_1 - 2,+}$ to $O(\frac{1}{\lambda_1}) v^{\lambda_1,+} + O(\frac{1}{\lambda_1}) v^{\lambda_1 - 2,+} + O(\lambda_1) v^{\lambda_1-4,+}$.. We finally see that from $(\lambda_1-2,\lambda_2+1)$ using $\sigma = (2,-1)$ we go from $v^{\lambda_1-2,+}$ to $O(\frac{1}{\lambda_1}) v^{\lambda_1,+} + O(1) v^{\lambda_1-2,+} + O(1) v^{\lambda_1-4,+}$ (with non-zero coefficients). Note that the coefficients always are non-zero except possibly for $(0,-1)$ coefficient of $v^{\lambda_1-2,+}$
We see we generate a module $V$ such that $I(\delta,\delta,\nu) / V$ is finite dimensional. Analyzing the infinitesimal character we see that the only finite-dimensional representations of $\mathfrak{sl}_{n+2}$ which appear as quotients of $I(\delta,\delta,\nu)$ have highest weight $a \omega_1 + b \omega_{n+1}$, all of which appear in Remark \ref{findimrep}. For all of these representations we have one of the relations in the proposition, so for $\nu$ not satisfying any of those relations $I(\delta,\delta,\nu) / V = 0$, and $V = I(\delta, \delta, \nu)$. This finishes the proof.
\end{proof}

We now consider other characters.

\begin{proposition}
\label{ireddothchar}
For $\nu_1 - \frac{n+2}{n} \nu_2 \notin 4 \Z_{\leq 0} \cup (2n+2) + 4 \Z_{\geq 0}$ and $\nu_1 + \frac{n+2}{n} \nu_2 \notin (-2 + 4 \Z_{\leq 0}) \cup (2n+4) + 4 \Z_{\geq 0}$ the space $I(1,0,\nu)$ is irreducible.

For $\nu_1 + \frac{n+2}{n} \nu_2 \notin 4 \Z_{\leq 0} \cup (2n+2) + 4 \Z_{\geq 0}$ and $\nu_1 - \frac{n+2}{n} \nu_2 \notin (-2 + 4 \Z_{\leq 0}) \cup (2n+4) + 4 \Z_{\geq 0}$ the space $I(0,1,\nu)$ is irreducible.
\end{proposition}

\begin{proof}
We first consider $I(1,0,\nu)$ and prove the $K$-type $V_{(1,0)}$ generates $I(1,0,\nu)$ and $I(1,0,\widetilde{\nu})$. We first see by Lemma \ref{oppsgnsothchar} that $V_{(1,0)}$ generates the $(1,2m)$ for $m \in \Z_{\geq 0}$ as there is no relation $\frac{n+2}{n} \nu_2 = (-1)^{k} (\nu_1 + 2 k)$ for some $k \in \Z_{\geq 0}$ for $\nu$ or $\widetilde{\nu}$. 
%Then we get $(1,1)$ by $\sigma = (0,1)$ or $\sigma = (0,-1)$ and thus all $(1,2n+1)$.
Then using Lemma \ref{linindep} we get all the $K$-types $V_{(\lambda_1, \lambda_2)}$ for all $\lambda_1$ when $\lambda_2$ is big and even. Using $T_{(1, \lambda_2)}^{(1, \lambda_2) + (0,-1)}(1,0,\nu)$ we obtain the $K$-types $V_{(1,\lambda_2)}$ for $\lambda_2$ odd and big enough. Now we proceed in a similar way as in Proposition \ref{repirrednu1nu2}. The proof for $I(0,1,\nu)$ is similar.
\end{proof}

\begin{remark}
\label{nis2case}
The proofs of Propositions \ref{repirrednu1nu2} and \ref{ireddothchar} do not quite work for $n=2$, as we only generate the $K$-types $(\lambda_1,\lambda_2)$ where $\lambda_2 \geq 0$, but in this case we also have $K$-types $-\frac{\lambda_1}{2} \leq \lambda_2 \leq 0$. Let $SU(2) \times SU(2) \rightarrow SO(4)$ be the double cover and denote by $V_a \boxtimes V_b$ the $SU(2) \times SU(2)$-representations, as in the proof of Lemma \ref{decompKlemma}; if $a + b \equiv 0 \mod 2$ these descend to representations of $SO(4)$. Let $\tau$ be the automorphism of $SL(4,\R)$ given by conjugation by $r = \mathrm{diag}(-1,1,1,1)$. Then $E$ and $F$ are stable under $\tau$ while for $SO(4)$ it exchanges $V_a \boxtimes V_b$ and $V_b \boxtimes V_a$. Moreover, it sends $l$-spherical vectors to $-l$-spherical vectors. It also preserves the group $P = MAN$ and preserves the characters $\chi_{\delta_1,\delta_2}$, so it induces an operator on the induced representation $U_{\tau}$ on $I(\delta_1,\delta_2,\nu)$ given by $f \mapsto (x \mapsto f(r x r))$, such that for $X \in \g$ we have $U_{\tau} X f = \tau(X) U_{\tau} f$. Hence the submodule generated by $V$ is stable under $U_{\tau}$ if the generating $K$-types are stable.

Now the $K$-types $\lambda_2 \geq 0$ correspond to the $K$-types $V_a \boxtimes V_b$ where $b \geq a$. As $E$, $F$ are invariant under $\tau$, if we generate all $V_a \boxtimes V_b$ where $b \geq a$ then we also generate all $V_a \boxtimes V_b$ where $b \leq a$ provided the generating $K$-type is invariant under $U_{\tau}$. This is true for $\chi_{0,0}$, $\chi_{1,0}$ and $\chi_{0,1}$ as the generating $K$-types are $V_0 \boxtimes V_0$, $V_1 \boxtimes V_1$ and $V_1 \boxtimes V_1$ respectively. For $\chi_{(1,1)}$ we took the generating $K$-type to be $V_{0,1} \cong V_0 \boxtimes V_2$. We see we can generate $V_{0,1}$ from $V_{2,0} \cong V_2 \boxtimes V_2$ by $\sigma = (-2,1)$, as from the formulas in \ref{valconstB} we see $B(\nu_1,(2,0),(-2,1),2,-1) \neq 0$. Hence we can take the generating $K$-type to be $V_{2,0}$, the generated submodule is invariant under $U_{\tau}$ and for all $K$-types $V_a \boxtimes V_b$ we also generate the $K$-types $V_b \boxtimes V_b$. With this remark, the proofs of Propositions \ref{repirrednu1nu2} and \ref{ireddothchar} work for $n=2$
\end{remark}

\subsection{The case $SL(3,\R)$}
\label{SL3case}

In the case $SL(3,\R)$ some of the theory we have developed breaks down and we can no longer use the same methods; we see for example that representations of $SO(3)$ are only indexed by one parameter $\lambda$ and the formulas of Theorem \ref{actiondiffop} will no longer work. In this case, however, $M$ is discrete and the characters are discrete series representations, which means we can use the theory developed by Speh and Vogan \cite{Speh81, SpehVog80}. In particular we apply \cite[Theorem 2.4]{Speh81} and obtain the following.

\begin{proposition}
\label{prop-casesl3}
The space $I(0,0,\nu)$ is reducible when $\nu_1 \in 1 + 2 \Z$ or $\pm \nu_1 + 3 \nu_2 \in 4 \Z$, and irreducible otherwise.

The space $I(1,1,\nu)$ is reducible when $\nu_1 \in 1 + 2 \Z$, or $\pm \nu_1 + 3 \nu_2 \in 2 +  4 \Z$ and $\nu_1 \pm 3 \nu_2 \neq 2$, and irreducible otherwise.

The space $I(1,0,\nu)$ is reducible when $\nu_1 \in 2 \Z$ and $\nu_2 \neq - 2$, $\nu_1 - 3 \nu_2 \in 4 \Z$, or $\nu_1 + 3 \nu_2 \in 2 +  4 \Z$ and $\nu_1 + 3 \nu_2 \neq 2$, and irreducible otherwise.

The space $I(0,1,\nu)$ is reducible when $\nu_1 \in 2 \Z$ and $\nu_2 \neq - 2$, $\nu_1 - 3 \nu_2 \in 2 + 4 \Z$ and $\nu_1 - 3 \nu_2 \neq 2$, or $\nu_1 + 3 \nu_2 \in 4 \Z$, and irreducible otherwise.
\end{proposition}

See \cite{Speh81} for a study of the unitarizability of these representations. See also \cite{Tor83} for another construction of unitary irreducible subrepresentations at points of reducibility, and see also \cite{KubOr25} for a study of representations of the nonlinear cover $\widetilde{SL}(3,\R)$.

\subsection{Unitarity and other representations of $SL(n+2,\R)$}

We obtain complementary series representations since the representations $I(\delta,\delta,\nu)$ are irreducible for $\nu_1 \in (n,n+2)$ and $\nu_2 \in \sqrt{-1} \R$.

\begin{corollary}
\label{complseries}
There is a complementary series $I(0,0,\nu)$ and $I(1,1,\nu)$ for $\nu_1 \in (n,n+2)$ and $\nu_2 \in \sqrt{-1} \R$.
\end{corollary}

\begin{proof}
This follows from Propositions \ref{firstredots} \ref{redptsnu1nu2}, \ref{repirrednu1nu2} and the existence of an intertwining operator $A(\delta,\delta,\nu):I(\delta,\delta,\nu) \rightarrow I(\delta,\delta,\widetilde{\nu})$. See also \cite[Proposition 16.4]{Kna86}.
\end{proof}

We can find our representations above inside of Vogan's classification \cite{Vog86}, and claim our complementary series comes from inducing the complementary series from $SL(2,\R)$. By the double induction formula we see that
$$ I(\delta_1,\delta_2,\nu) \cong Ind_{M'A'N'}^{G} \left( Ind_{M A'' N''}^{M'} ( \chi_{\delta_1,\delta_2}  \otimes e^{ (\nu_1 - n) E^*} \otimes 1) \otimes e^{\nu_2 F^*} \otimes 1 \right),$$
where $P = MAN$ is the Heisenberg parabolic, $\m' \oplus \af' \oplus \n' \cong (\mathfrak{sl}_2(\R), \mathfrak{sl}_n(\R)) \oplus \R F \oplus \n_1$, $\af'' \cong \R E$ and $\n'' \cong \n_2$. Furthermore, if $M_0'$ is the connected component of $\m'$ then $M' = M_0' \cup w_0 M_0'$. By the restriction map $f \mapsto f|_{SL(2,\R)^{\pm}}$ we can identify $Ind_{M''A''N''}^{M'} ( \sigma \otimes e^{ (\nu_1 - n) E^*} \otimes 1)$ with $Ind_{M_1 A_1 N_1}^{SL(2,\R)^{\pm}}( \sigma|_{SL(2,\R)^{\pm}} \otimes e^{(\nu_1 - n) E^*} \otimes 1)$, where $M_1 = \{I, w_0, \mathrm{diag}(-I_2,I_n), w_0 \mathrm{diag}(-I_2,I_n) \}$, we identify $E$ with $\begin{pmatrix} 1 & 0 \\ 0 & -1 \end{pmatrix} \in \mathfrak{sl}(2,\R)$ and $N_1$ are the upper triangular matrices in $SL_2(\R)$. For $\sigma = \chi_{0,0}$, $\chi_{1,1}$ we have $\sigma(-I_2) = 1$, so we have a complementary series for $0 < \nu_1 - n < 2$ when $\nu_2 \in \sqrt{-1} \R$.
%, where $0 < \nu_1 -n < 1$ is isomorphic to the value for $2 - \nu_1 + n$.
So we find this in Vogan's classification, as $M'A'N'$ is the parabolic generated by this $SL(2,\R) \times SL(n,\R)$; this corresponds to an almost spherical representation \cite[Definition 3.7]{Vog86} associated to a partition $(2,n)$ of $n+2$, the $\tau_1$ a complementary series of $SL(2,\R)$ and $\tau_2$ a unitary character of $GL(n,\R)$. See also \cite{dahl19} for explicit calculations of the Knapp-Stein kernel and an analysis through Fourier analysis, also at the endpoints of the complementary series.

%\begin{corollary}
%\label{heisparunit}
%The Harish-Chandra module of the space of smooth solutions of the Heisenberg ultrahyperbolic equation as introduced in \cite{Kab12sln2} is a unitarizable $(\g,K)$-module when induced from the characters $\chi_{0,0}$, $\chi_{1,1}$ for $M$ and $\nu_2 \in \sqrt{-1} \R$.
%\end{corollary}

%Specifically, we have an irreducible subrepresentation of our parabolically induced representations, and as we can embed our induced represen by \cite[Corollary 5.9]{vdBan97} it is unique. Sketch: we have our parabolic subgroup $MAN$ where $\m \cong \mathfrak{sl}(n,\R)$. Then we find a minimal parabolic subgroup of $M$, which we call $M'A'N'$. Then $M'AA'NN'$ is a minimal parabolic subgroup of $SL(n+2,\R)$ and by the double induction formula $Ind_{M'AA'NN'} \cong Ind_{MAN}^G Ind_{M'A'N'}^M$. Choosing the right parameters we find that $Ind_{MAN} \sigma \subseteq Ind_{MAN}^G Ind_{M'A'N'}$. We know that Kable's representation is the unique irred subrep; now our subrep is there as well. for $\nu_2 \in i \R$ this is unitary.

We compare our results with those obtained in \cite{Fra22}. Frahm finds $L^2$-models for representations of $G$ as subrepresentations of $\bigoplus_{\delta_1,\delta_2} I(\delta_1,\delta_2,\nu)$. He constructs two subrepresentations for $\nu_2 = \pm \frac{n \nu_1}{n+2}$ as a kernel of a conformally invariant system of second order differential operators as in \cite{BKZ08}, and his constant $r$ corresponds to $\nu_1 - \frac{n+2}{2}$, while $\nu_2 = \pm \frac{n}{n+2} \nu_1$. %$\nu_1 = \frac{n+2}{2} + r$, where $\nu_2 = \frac{n}{2} + \frac{n}{n+2} r$.
For $r = \nu_1 - \frac{n+2}{2} \in \sqrt{-1} \R$ he proves these representations are irreducible and unitarizable for $n \geq 2$. These correspond to minimal representations in a sense, which Tamori \cite{Tam23} makes precise. He constructs $(\g,\kf)$-modules whose associated varieties are the closure of the minimal nilpotent coadjoint orbit. He defines $a$-minimal representations and finds these representations as subrepresentations of a representation induced from a maximal parabolic subgroup $M'A'N'$, where $\m' \cong \mathfrak{sl}_{n+1}(\R)$ is embedded in the lower $(n+1)$-block of $G$, $E_1 = E+F$, $\af' = \R E_1$ and $\n' \cong \R^{n+1}$. Let $F_1 = E - \frac{n+2}{n} F$, then using the double induction formula we see
$$ I(\delta_1,\delta_2,\nu) \cong Ind_{M'A'N'}^{G} \left( Ind_{M A'' N''}^{M'}(\chi_{\delta_1,\delta_2} \otimes e^{(\nu_1 - \frac{n+2}{n} \nu_2)(F_1)^*} \otimes 1 ) \otimes e^{(\nu_1 + \nu_2)E_1^*} \otimes 1 \right),$$
where $\af'' = \R F_1$. We see that if $\nu_1 = \frac{n+2}{n} \nu_2$ we obtain a subrepresentation of our induced representation which is induced from the maximal parabolic for $\chi_{0,0}$ and $\chi_{1,0}$. Tamori \cite[Theorem 5.3]{Tam23} then finds a subrepresentation which corresponds to the minimal representation for type $A$. Observe that this is a similar condition as from \cite{Fra22}. Tamori's constant $a$ corresponds to $\frac{n+2}{2} - \nu_1$, and again he states the representation is unitarizable for $a \in \sqrt{-1} \R$.

We now compare our results with those obtained by Kable \cite{Kab12sln2}, focusing on the case $n \geq 2$. Note that he does not explicitly identify the character of $M$ for his subrepresentations. For the parameter $\nu_1 = n$ and $\nu_2$ arbitrary he defines a subrepresentation of $I(\delta_1,\delta_2,\nu)$ as a solution to the Heisenberg ultrahyperbolic equation, which is the kernel of a differential operator. He then determines the $K$-types of the sum of these subrepresentations for each character of $M$, i.e. in $\bigoplus I_{\delta_1,\delta_2} I(\delta_1,\delta_2,\nu)$; we remark that he only proves the sum of these subrepresentations is non-zero and has positive multiplicity for some $K$-types. His $K$-types where $a-|b|$ are contained in subrepresentations that appear at our reduction points $\nu_1 = n$ and arbitrary $\nu_2$ for the characters $\chi_{0,0}$, $\chi_{1,1}$; that the subrepresentations are non-zero for both these characters can be read off as the multiplicity of the $K$-types $(0,0)$ and $(0,1)$, which only appear for the characters $\chi_{0,0}$ and $\chi_{1,1}$  respectively, are one. For the character $\chi_{1,0}$ we obtain by the reducibility that there are subrepresentations which appear at the points $\nu_1 = n$, $\frac{n+2}{n} \nu_2 \in n + 4 \Z_{\geq 0}$ or $- \frac{n+2}{n} \nu_2 \in n + 2+ 4 \Z_{\geq 0}$. For $\chi_{0,1}$ there are subrepresentations which appear at the points $\nu_1 = n$, $\frac{n+2}{n} \nu_2 \in n + 2 + 4 \Z_{\geq 0}$ or $- \frac{n+2}{n} \nu_2 \in n + 4 \Z_{\geq 0}$. These are precisely the discrete points in \cite[Theorem 5.12]{Kab12sln2}; the $K$-types that appear for these characters are exactly his $a - |b|$ odd. As the reducibility points for the characters $\chi_{0,1}$, $\chi_{1,0}$ do not overlap, we find that we can identify subrepresentations of the induced representation at the reducibility points for these characters as solutions to the Heisenberg ultrahyperbolic equation. For $\nu_2 \in \sqrt{-1} \R^*$ and $n \geq 3$ the representations Kable defines are irreducible and unitarizable by \cite[Theorem 6.6]{Kab12sln2}, as they lie on the unitary axis in \cite{HoweLee99}.

\newpage
\appendix

\section{Coefficients for branching}
\label{AppCoeff}

We gather some lemmas here. Recall the constants $C^i(\sigma,\lambda,p)$ from Corollary \ref{multformsph}.

\begin{lemma}
\label{oppsigns}
The constants
$$C^0((2,0),(0,\lambda_2),0)C^0((-2,2),(2,\lambda_2),0)$$
and 
$$C^1((2,0),(0,\lambda_2),0)C^{-1}((-2,2),(2,\lambda_2),2)$$
have opposite signs, i.e.
$$\frac{C^0((2,0),(0,\lambda_2),0)C^0((-2,2),(2,\lambda_2),0)}{|C^0((2,0),(0,\lambda_2),0)C^0((-2,2),(2,\lambda_2),0)|} = - \frac{C^1((2,0),(0,\lambda_2),0)C^{-1}((-2,2),(2,\lambda_2),2)}{|C^1((2,0),(0,\lambda_2),0)C^{-1}((-2,2),(2,\lambda_2),2)|}.$$
\end{lemma}

\begin{proof}
We assume we have made a choice of unit vectors $v_{\lambda}^p$. We denote by $P_{\lambda}$ the projection onto $V_{\lambda}$, where we do not specify the domain where no confusion is possible. Note that as $V_{\lambda}$ will always appear with multiplicity one, so it is well-defined. We will often implicitly use the fact that $\omega_1$ is a (quasi-)minuscule weight, so we explicitly know the decomposition of the tensor product $\C^{n+2} \otimes V_{\lambda}$. We have
\begin{flalign*}
& C^0((2,0),(0,\lambda_2),0)C^0((-2,2),(2,\lambda_2),0)
\\ & = \langle P_{(0,\lambda_2+2)} \left( e_- \otimes P_{(1,\lambda_2+1)}(e_+ \otimes P_{(2,\lambda_2)} (e_- \otimes P_{(1,\lambda_2)}( e_+ \otimes v^0_{(0,\lambda_2)} ) )) \right), v_{(0,\lambda_2+2)} \rangle,
\end{flalign*}
%We note that the whole time we need to study a decomposition, which is as follows.
%$$(0, \lambda_2) \rightarrow (1, \lambda_2) \rightarrow (2, \lambda_2) \rightarrow (1, \lambda_2 + 1) \rightarrow (0,\lambda_2+2).$$
%The coefficient $C^0((2,0),(0,\lambda_2),0)C^0((-2,2),(2,\lambda_2),0)$ is tied  to
%$$P_{(0,\lambda_2+2)} \left( e_- \otimes P_{(1,\lambda_2+1)}(e_+ \otimes P_{(2,\lambda_2)} (e_- \otimes P_{(1,\lambda_2)}( e_+ \otimes v^0 ) )) \right),$$
while 
%the coefficient $C^1((2,0),(0,\lambda_2),0)C^{-1}((-2,2),(2,\lambda_2),2)$ comes from
%$$P_{(0,\lambda_2+2)} \left( e_- \otimes P_{(1,\lambda_2+1)}(e_- \otimes P_{(2,\lambda_2)} (e_+ \otimes P_{(1,\lambda_2)}( e_+ \otimes v^0 ) )) \right).$$
\begin{flalign*}
& C^1((2,0),(0,\lambda_2),0)C^{-1}((-2,2),(2,\lambda_2),2)
\\ & = \langle P_{(0,\lambda_2+2)} \left( e_- \otimes P_{(1,\lambda_2+1)}(e_- \otimes P_{(2,\lambda_2)} (e_+ \otimes P_{(1,\lambda_2)}( e_+ \otimes v^0_{(0,\lambda_2)} ) )) \right), v_{(0,\lambda_2+2)} \rangle.
\end{flalign*}
We now explicitly construct the map $P_{(1,\lambda_2+1)}$. Let
$$ P: \C^{n+2} \otimes \left( \bigodot^2 \C^{n+2} \right)_{\mathrm{trace-free}} \otimes V_{(0,\lambda_2)} \rightarrow (\bigwedge^2 \C^{n+2}) \otimes \C^{n+2} \otimes V_{(0,\lambda_2)} $$
be the $SO(n+2)$-intertwining map defined by
$$ v_1 \otimes (v_2 \odot v_3) \otimes u \mapsto \left( v_1 \wedge v_2 \right) \otimes v_3 \otimes u + \left( v_1 \wedge v_3 \right) \otimes v_2 \otimes u.$$
Let $Q_{(1,\lambda_2+1)}$ be the projection
$$Q_{(1,\lambda_2+1)}: (\bigwedge^2 \C^{n+2}) \otimes \C^{n+2} \otimes V_{(0,\lambda_2)} \rightarrow V_{(1,\lambda_2+1)},$$
%up to a constant this is the projection from $\C^{n+2} \otimes \left( \bigodot^2 \C^{n+2} \right)_{\mathrm{trace} \ 0} \otimes V_{(0,\lambda_2)}$ (as $(e_3 + i e_4) \otimes (e_+ \odot e_+) \otimes v_{(0,\lambda_2)}^{\mathrm{h.w.}}$ has non-zero projection).
We claim that (up to constants)
\begin{flalign*}
& P_{(1,\lambda_2+1)}(e_{\pm} \otimes P_{(2,\lambda_2)} (e_{\mp} \otimes P_{(1,\lambda_2)}( e_+ \otimes v^0_{(0,\lambda_2)} ) )) = Q_{(1,\lambda_2+1)} P( e_{\pm} \otimes (e_{\mp} \odot e_+) \otimes v^0_{(0,\lambda_2)}).
\end{flalign*}

To see this fact, observe that as $(2,\lambda_2)$ appears with multiplicity one both in $(\C^{n+2})^{\otimes 2} \otimes V_{(0,\lambda_2)}$ and $(\C^{n+2})^{\odot 2}_{\mathrm{trace-free}} \otimes V_{(0,\lambda_2)} \subseteq (\C^{n+2})^{\otimes 2} \otimes V_{(0,\lambda_2)},$ it follows that
%$V_{(1,\lambda_2+1)}$ appears in $ (\bigwedge^2 \C^{n+2}) \otimes \C^{n+2} \otimes V_{(0,\lambda_2)}$ with multiplicity one. Now we embed $(2,\lambda_2)$ in $(\C^{n+2})^{\otimes 2} \otimes V_{(0,\lambda_2)}$. We then see that as $(2,\lambda_2)$ appears with multiplicity one in the decomposition, that
$$P_{(2,\lambda_2)} (e_{\pm} \otimes P_{(1,\lambda_2)}( e_+ \otimes v ) ) = P_{(2,\lambda_2)}(e_{\pm} \otimes e_+ \otimes v) = P_{(2,\lambda_2)}(e_{\pm} \odot e_+ \otimes v),$$
where by abuse of notation we also denote $P_{(2,\lambda_2)}$ the projection to $V_{(2,\lambda_2)}$ from $\bigodot^2 \C^{n+2} \otimes V_{(0,\lambda_2)}$. We also note that as $(\C^{n+2})^{\otimes 3} \otimes V_{(0,\lambda_2)}$ contains $V_{(1,\lambda_2+1)}$ with multiplicity two and we have a decomposition
\begin{flalign*}
& (\C^{n+2})^{\otimes 3} \otimes V_{(0,\lambda_2)} \cong \C^{n+2} \otimes (\C^{n+2})^{\otimes 2} \otimes V_{(0,\lambda_2)}
\\ & \cong \C^{n+2} \otimes (\C^{n+2})^{\odot 2}_{\mathrm{trace-free}} \otimes V_{(0,\lambda_2)} \oplus \C^{n+2} \otimes (\C^{n+2})^{\wedge 2} \otimes V_{(0,\lambda_2)} \oplus \C^{n+2} \otimes \C \otimes V_{(0,\lambda_2)},
\end{flalign*}
that $\C^{n+2} \otimes (\C^{n+2})^{\odot 2}_{\mathrm{trace-free}} \otimes V_{(0,\lambda_2)}$ contains $V_{(1,\lambda_2+1)}$ with multiplicity one.
%The way we see that is the following: the highest weight of $(1,\lambda_2+1)$ is $2 e_1 + e_2 + \lambda_2 (e_1 + e_2)$. So in $(\C^{n+2})^{\otimes 3} \otimes V_{(0,\lambda_2)}$ it appears with multiplicity $2$ (with the Weyl group we move $e_1 \mapsto e_1$ and then one of the $e_1 \mapsto e_2$, using the quasiminuscule property). We also know that
%$$(\C^{n+2})^{\otimes 3} \otimes V_{(0,\lambda_2)} \cong \C^{n+2} \otimes (\C^{n+2})^{\otimes 2} \otimes V_{(0,\lambda_2)} \cong \C^{n+2} \otimes (\C^{n+2})^{\odot 2} \otimes V_{(0,\lambda_2)} \oplus \C^{n+2} \otimes (\C^{n+2})^{\wedge 2} \otimes V_{(0,\lambda_2)} \oplus \C^{n+2} \otimes \C \otimes V_{(0,\lambda_2)},$$
%and the representation $V_{(1,\lambda_2+1)}$ appears with multiplicity one in $\C^{n+2} \otimes (\C^{n+2})^{\wedge 2} \otimes V_{(0,\lambda_2)}$, hence it also appears with multiplicity one in $\C^{n+2} \otimes (\C^{n+2})^{\odot 2} \otimes V_{(0,\lambda_2)}$.
It follows that, up to a constant, $Q_{(1,\lambda_2+1)} P$ is the unique projection onto $V_{(1,\lambda_2+1)}$ and
%We define a map
%$$ P: \C^{n+2} \otimes \left( \bigodot^2 \C^{n+2} \right)_{\mathrm{trace} \ 0} \otimes V_{(0,\lambda_2)} \rightarrow (\bigwedge^2 \C^{n+2}) \otimes \C^{n+2} \otimes V_{(0,\lambda_2)} $$
%by
%$$ v_1 \otimes (v_2 \odot v_3) \otimes u \mapsto \left( v_1 \wedge v_2 \right) \otimes v_3 \otimes u + \left( v_1 \wedge v_3 \right) \otimes v_2 \otimes u.$$
%If we compose with the projection onto
%$$Q_{(1,\lambda_2+1)}: (\bigwedge^2 \C^{n+2}) \otimes \C^{n+2} \otimes V_{(0,\lambda_2)} \rightarrow V_{(1,\lambda_2+1)},$$
%up to a constant this is the projection from $\C^{n+2} \otimes \left( \bigodot^2 \C^{n+2} \right)_{\mathrm{trace} \ 0} \otimes V_{(0,\lambda_2)}$ (as $(e_3 + i e_4) \otimes (e_+ \odot e_+) \otimes v_{(0,\lambda_2)}^{\mathrm{h.w.}}$ has non-zero projection).Then we see that
\begin{flalign*}
& P_{(1,\lambda_2+1)}(e_+ \otimes P_{(2,\lambda_2)} (e_- \otimes P_{(1,\lambda_2)}( e_+ \otimes v^0_{(0,\lambda_2)} ) )) = Q_{(1,\lambda_2+1)} P( e_+ \otimes (e_- \odot e_+) \otimes v^0_{(0,\lambda_2)}) \\ & = Q_{(1,\lambda_2+1)}( (e_+ \wedge e_-) \otimes e_+ \otimes v^0_{(0,\lambda_2)} ),
\end{flalign*}
while
\begin{flalign*}
& P_{(1,\lambda_2+1)}(e_- \otimes P_{(2,\lambda_2)} (e_+ \otimes P_{(1,\lambda_2)}( e_+ \otimes v^0_{(0,\lambda_2)} ) )) = Q_{(1,\lambda_2+1)} P( e_- \otimes (e_+ \odot e_+) \otimes v^0_{(0,\lambda_2)})
\\ & = 2 Q_{(1,\lambda_2+1)} ((e_- \wedge e_+) \otimes e_+ \otimes v^0_{(0,\lambda_2)}).
\end{flalign*}
The lemma follows.
\end{proof}

We prove another lemma in a similar fashion.

\begin{lemma}
\label{oppsgnsothchar}
We have
\begin{flalign*}
& T_{(3,\lambda_2)}^{(1, \lambda_2+2)}(1,0,\nu)T_{(1,\lambda_2)}^{(3,\lambda_2)}(1,0,\nu) v^{1,+}_{(1,\lambda_2)}
\\ & = \left( (\nu_1 + 2 \lambda_2 + 3 + (-1)^{\lambda_2} 2 \frac{n+2}{n} \nu_2)(\nu_1 + 2 \lambda_2 - 1 - (-1)^{\lambda_2} 2 \frac{n+2}{n} \nu_2) + 3(\nu_1 + 2 \lambda_2 + 1)^2 \right) C_1 v^{1,+}_{(1,\lambda_2)},
\end{flalign*}
and
\begin{flalign*}
& T_{(3,\lambda_2)}^{(1, \lambda_2+2)}(0,1,\nu)T_{(1,\lambda_2)}^{(3,\lambda_2)}(0,1,\nu) v^{1,-}_{(1,\lambda_2)}
\\ & = \left( (\nu_1 + 2 \lambda_2 + 3 - (-1)^{\lambda_2} 2 \frac{n+2}{n} \nu_2)(\nu_1 + 2 \lambda_2 - 1 + (-1)^{\lambda_2} 2 \frac{n+2}{n} \nu_2) + 3(\nu_1 + 2 \lambda_2 + 1)^2 \right) C_2 v^{1,-}_{(1,\lambda_2)}
\end{flalign*}
for some constants non-zero $C_1, C_2 \in \C$.
\end{lemma}

\begin{proof}
We prove the fact for the character $\chi_{(1,0)}$; the case $\chi_{(0,1)}$ is similar. The proof essentially boils down to a more complicated version of Lemma \ref{oppsigns}. Note that for ease of notation we normalize $v^1_{(1,\lambda_2)}$ such that $v^1_{(1,\lambda_2)} = P_{(1,\lambda_2)}(e_+ \otimes v^0_{(0,\lambda_2)})$; $v^1_{(1,\lambda_2)}$ will not be a unit vector, but it will make the equations cleaner. By Corollary \ref{multformsph} and Lemma \ref{valconstB}
\begin{flalign*}
& T_{(1,\lambda_2)}^{(3,\lambda_2)}(1,0,\nu) v^{1,+}_{(1,\lambda_2)} = \left( (\nu_1 + 2 \lambda_2 + 3) C^{-1}( (2,0), (1,\lambda_2), 1) + 2 \frac{\nu_2 (n+2)}{n} C^0( (2,0), (1,\lambda_2), 1) \right) v^{1,+}_{(3,\lambda_2)} +
\\ & (\nu_1 + 1 + 2 \lambda_2) C^{1}( (2,0), (1,\lambda_2), 1) v^{3,+}_{(3,\lambda_2)}.
\end{flalign*}
By Proposition \ref{explicitcfctn} for a non-zero constant $c_1 \in \R$
$$ |C^{-1}( (2,0), (1,\lambda_2),1)|^2 = 8 (4+ 2 \lambda_2 + a) (2 + 2 \lambda_2 + a) c_1,$$
$$ |C^0((2,0),(1,\lambda_2),1)|^2 = 8 (4+2 \lambda_2 + a)(2 + 2 \lambda_2 + a) c_1,$$
and
$$ |C^1((2,0),(1,\lambda_2),1)|^2 = 24 (4 + 2 \lambda_2 + a)(6 + 2 \lambda_2 + a) c_1. $$
Similarly,
\begin{flalign*}
& T_{(3,\lambda_2)}^{(1,\lambda_2+2)}(1,0,\nu)(v^{3,+}) = (\nu_1 + 1 + 2 \lambda_2) C^{-1}((-2,2),(3,\lambda_2),3) v^{1,+},
\end{flalign*}
where there is a constant $c_2 \in \R$ such that
$$|C^{-1}((-2,2),(3,\lambda_2),3)|^2 = 24(4+2 \lambda_2 + a)(2 + 2 \lambda_2 + a) c_2.$$
Also,
\begin{flalign*}
& T_{(3,\lambda_2)}^{(1,\lambda_2+2)}(1,0,\nu)(v^{1,+}) = \left( 2 \frac{\nu_2 (n+2)}{n} C^0( (-2,2), (3, \lambda_2), 1) + (\nu_1 + 2 \lambda_2 - 1) C^{-1}((-2,2),(3,\lambda_2),1) \right) v^{1,+},
\end{flalign*}
where for the same constant $c_2$
$$|C^{0}((-2,2),(3,\lambda_2),1)|^2 = 8 (4+2 \lambda_2 + a)(6 + 2 \lambda_2 + a) c_2$$
and
$$|C^{-1}((-2,2),(3,\lambda_2),1)|^2 = 8(4+2 \lambda_2 + a)(6 + 2 \lambda_2 + a) c_2.$$
We simplify these expressions.

We note that
$$C^{-1}( (2,0), (1,\lambda_2),1) = \langle P_{(3,\lambda_2)} (w_0 (e_- \otimes e_- \otimes v^{1}_{(1, \lambda_2)}) ), v^1_{(3,\lambda_2)} \rangle,$$
and
$$C^0( (2,0), (1,\lambda_2), 1) = \langle P_{(3,\lambda_2)} (e_+ \otimes e_- \otimes v^1_{(1,\lambda_2)}), v^1_{(3,\lambda_2)} \rangle.$$
Note that embedding $V_{(1,\lambda_2)} \subseteq \C^{n+2} \otimes V_{(0,\lambda_2)}$ we obtain
\begin{flalign*}
& P_{(3,\lambda_2)} (w_0 (e_- \otimes e_- \otimes v^{1}_{(1, \lambda_2)}) ) = w_0 P_{(3,\lambda_2)} ( e_- \otimes e_- \otimes e_+ \otimes v^{0}_{(0, \lambda_2)})
\\ & = w_0 P_{(3,\lambda_2)} ( ( e_+ e_-^{\odot 2} ) \otimes v^{0}_{(0, \lambda_2)}) = (-1)^{\lambda_2} P_{(3,\lambda_2)}(e_+ \otimes e_- \otimes v^1_{(1,\lambda_2)}).
\end{flalign*}
\begin{comment}
$$P_{(3,\lambda_2)}( e_+ \otimes e_- \otimes (e_+ \otimes v^0_{(0,\lambda_2)})) = P(e_+^{\odot 2} e_- \otimes v^0_{(0,\lambda_2)})$$
As $w_0 v^0_{(0,\lambda_2)} = (-1)^{\lambda_2} v^0_{(0,\lambda_2)}$, we realize that $ \langle P_{(3,\lambda_2)} ( w_0 (e_- \otimes e_- \otimes v^{1}_{(1, \lambda_2)}) ), e_+ \otimes e_- \otimes v^1_{(1,\lambda_2)} \rangle = (-1)^{\lambda_2} $ so we realize that
\end{comment}
It follows that
\begin{equation}
\label{firsteqapp}
\begin{split}
T_{(1,\lambda_2)}^{(3,\lambda_2)}(1,0,\nu) v^{1,+} = & \left( (\nu_1 + 2 \lambda_2 + 3)  + (-1)^{\lambda_2} 2 \frac{\nu_2 (n+2)}{n} \right) 2 \sqrt{2} \sqrt{(4 + 2 \lambda_2 + a)(2 + 2 \lambda_2 + a)} c_3 v^{1,+}
\\ & + (\nu_1 + 1 + 2 \lambda_2) 2 \sqrt{6} \sqrt{(4 + 2 \lambda_2 + a)(6 + 2 \lambda_2 + a)} c_4 v^{3,+}
\end{split}
\end{equation}
for some $c_3, c_4 \in \C$ such that $|c_3| = |c_4| = \sqrt{c_1}$.

We observe that
$$C^{0}((-2,2),(3,\lambda_2),1) = \langle P_{(1,\lambda_2+2)}\left(e_+ \otimes P_{(2,\lambda_2 + 1)}(e_- \otimes v^1_{(3,\lambda_2)}) \right), v^1_{(1,\lambda_2+2)} \rangle$$
and
$$C^{-1}((-2,2),(3,\lambda_2),1) = \langle P_{(1,\lambda_2+2)} \left( e_+ \otimes P_{(2,\lambda_2 + 1)}(e_+ \otimes w_0 v^{1}_{(3,\lambda_2)}) \right), v^1_{(1,\lambda_2+2)} \rangle.$$
As
$$ P_{(2,\lambda_2+1)}(e_+ \otimes w_0 v^{1}_{(3,\lambda_2)}) = w_0 P_{(2,\lambda_2+1)}(e_- \otimes v^{1}_{(3,\lambda_2)}) = (-1)^{\lambda_2 + 1} P_{(2,\lambda_2+1)}(e_- \otimes v^{1}_{(3,\lambda_2)}),$$
it follows that
\begin{equation}
\label{secondeqapp}
T_{(3,\lambda_2)}^{(1,\lambda_2+2)}(1,0,\nu)(v^{1,+}) = \left( (\nu_1 + 2 \lambda_2 - 1) + (-1)^{\lambda_2 + 1} 2 \frac{\nu_2 (n+2)}{n}  \right) 2 \sqrt{2} \sqrt{(4 + 2 \lambda_2 + a)(6 + 2 \lambda_2 + a)} c_5 v^{1,+}.
\end{equation}
for some constant $c_5 \in \C$ such that $|c_5| = \sqrt{c_2}$.

To prove the lemma we will show that 
$$ \langle P_{(1, \lambda_2+2)}\left( e_+ \otimes e_+ \otimes P_{(3,\lambda_2)} \left( e_- \otimes e_- \otimes v^{1}_{(1,\lambda_2)} \right) \right), v^1_{(1,\lambda_2+2)} \rangle$$
and
$$\langle P_{(1,\lambda_2+2)}\left( e_- \otimes e_- \otimes P_{(3, \lambda_2)} \left( e_+ \otimes e_+ \otimes v^{1}_{(1,\lambda_2)} \right) \right), v_{(1,\lambda_2+2)}^1 \rangle$$
have the same sign, in a similar fashion as in Lemma \ref{oppsigns}. Note that by Lemma \ref{valconstB} and Equations \ref{firsteqapp} and \ref{secondeqapp} this will prove the lemma, again as in Lemma \ref{oppsigns}.
%$C^{-1}((-2,2),(3,\lambda_2),3)$ comes from $P_{(1,\lambda_2+2)}\left( e_+ \otimes e_+ P_{(3, \lambda_2)} \left( e_- \otimes e_- \otimes v^{-1} \right) \right)$. We now claim this has the same sign as $P_{(1,\lambda_2+2)}\left( e_+ \otimes e_+ P_{(3, \lambda_2)} \left( e_- \otimes e_- \otimes v^{1} \right) \right)$.

We embed $V_{(1,\lambda_2)}$ into $\C^{n+2} \otimes V_{(0,\lambda_2)}$ and study the projections of $(\C^{n+2})^{\otimes 4} \otimes \C^{n+2} \otimes V_{(0,\lambda_2)}$ onto $V_{(1,\lambda_2 + 2)}$. We note that $V_{(3,\lambda_2)}$ appears in $( \C^{n+2})^{\otimes 3} \otimes V_{(0,\lambda_2)}$ with multiplicity one, and that $V_{(2,\lambda_2+1)}$ appears in $(\C^{n+2})^{\otimes 4} \otimes V_{(0,\lambda_2)}$ with multiplicity three. Observe that
\begin{flalign*}
& P_{(1, \lambda_2+2)}\left( e_{\pm} \otimes e_{\pm} \otimes P_{(3,\lambda_2)} \left( e_{\mp} \otimes e_{\mp} \otimes v^{1}_{(1,\lambda_2)} \right) \right)
\\ & = P_{(1, \lambda_2+2)}\left( e_{\pm} \otimes P_{(2,\lambda_2+1)} \left( e_{\pm} \otimes P_{(3,\lambda_2)} \left( e_{\mp} \otimes P_{(2,\lambda_2)} \left( e_{\mp} \otimes e_+ \otimes v^0_{(0,\lambda_2)} \right) \right) \right) \right)
\\ & = P_{(1, \lambda_2+2)}\left( e_{\pm} \otimes P_{(2,\lambda_2+1)} \left( e_{\pm} \otimes P_{(3,\lambda_2)} \left( e_{\mp}^{\odot 2}e_+ \otimes v^0_{(0,\lambda_2)} \right) \right) \right).
\end{flalign*}

We claim $\C^{n+2} \otimes (\C^{n+2})^{\odot 3} \otimes V_{(0,\lambda_2)}$ contains $V_{(2,\lambda_2+1)}$ with multiplicity one. The way to see this is as follows. We consider $(\C^{n+2})^{\odot m} \subseteq (\C^{n+2})^{\otimes m}$ in the regular way and we see that
$$ \C^{n+2} \otimes (\C^{n+2})^{\odot 3} \otimes V_{(0,\lambda_2)} \subseteq (\C^{n+2})^{\otimes 2} \otimes (\C^{n+2})^{\odot 2} \otimes V_{(0,\lambda_2)} \subseteq (\C^{n+2})^{\otimes 4} \otimes V_{(0,\lambda_2)}.$$
As $\omega_1$ is (quasi-)minuscule, we see the representation $(\C^{n+2})^{\otimes 4} \otimes V_{(0,\lambda_2)}$ contains $V_{(2,\lambda_2+1)}$ with multiplicity three. From the decomposition
$$( \C^{n+2})^{\otimes 4} \otimes V_{(0,\lambda_2)} \cong (\C^{n+2})^{\otimes 2} \otimes \left( (\C^{n+2})^{\odot 2} \oplus (\C^{n+2})^{\wedge 2} \right) \otimes V_{(0,\lambda_2)}$$
we see $(\C^{n+2})^{\otimes 2} \otimes (\C^{n+2})^{\odot 2} \otimes V_{(0,\lambda_2)}$ contains $V_{(2,\lambda_2+1)}$ with multiplicity at most two. Furthermore, there exists an intertwiner
$$ Q: (\C^{n+2})^{\otimes 2} \otimes (\C^{n+2})^{\odot 2} \otimes V_{(0,\lambda_2)} \rightarrow \C^{n+2} \otimes (\C^{n+2})^{\wedge 2} \otimes \C^{n+2} \otimes V_{(0,\lambda_2)}.$$
given by
$$ u \otimes \left( v_1 \otimes (v_2 \odot v_3) \otimes v_4 \right) \mapsto u \otimes \left( ((v_1 \wedge v_2) \otimes v_3 + (v_1 \wedge v_3) \otimes v_2) \otimes v_4 \right),$$
where $Q(\C^{n+2} \otimes (\C^{n+2})^{\odot 3} \otimes V_{(0,\lambda_2)}) = 0$. As $e_1 \otimes (e_1 \wedge e_2) \otimes e_1 \otimes v_{(0,\lambda_2)} \in \mathrm{Im}(Q)$ and there is an obvious non-zero projection $P_{(2,\lambda_2+1)}: \C^{n+2} \otimes (\C^{n+2})^{\wedge 2} \otimes \C^{n+2} \otimes V_{(0,\lambda_2)} \rightarrow V_{(2,\lambda_2+1)}$ such that
$$P_{(2,\lambda_2+1)}(e_1 \otimes (e_1 \wedge e_2) \otimes e_1 \otimes v_{(0,\lambda_2)}) \neq 0,$$
we conclude that $\C^{n+2} \otimes (\C^{n+2})^{\odot 3} \otimes V_{(0,\lambda_2)}$ indeed contains $V_{(2,\lambda_2+1)}$ with multiplicity one.

Now we construct an intertwining map $\C^{n+2} \otimes (\C^{n+2})^{\odot 3} \otimes V_{(0,\lambda_2)} \rightarrow V_{(2,\lambda_2+1)}$. Consider
$$Q': \C^{n+2} \otimes (\C^{n+2})^{\odot 3} \otimes V_{(0,\lambda_2)} \rightarrow (\C^{n+2})^{\odot 2} \otimes (\C^{n+2})^{\wedge 2} \otimes V_{(0,\lambda_2)}$$
given by
$$ v_1 \otimes (v_2 v_3 v_4) \otimes v_5 \mapsto \left( (v_1 \wedge v_2) \otimes v_3 v_4 + (v_1 \wedge v_3) \otimes v_2 v_4 + (v_1 \wedge v_4) \otimes v_2 v_3 \right) \otimes v_5.$$
Then up to a constant $P_{(2, \lambda_2+1)} Q'$ is the unique projection from $\C^{n+2} \otimes (\C^{n+2})^{\odot 3} \otimes V_{(0,\lambda_2)}$ to $V_{(2,\lambda_2+1)}$.

Now by similar arguments as previously we realize $\C^{n+2} \otimes (\C^{n+2})^{\odot 2} \otimes (\C^{n+2})^{\wedge 2} \otimes V_{(0,\lambda_2)}$ contains $V_{(1,\lambda_2+2)}$ with multiplicity one. We consider the intertwiner
\begin{comment}
$$Q'': \C^{n+2} \otimes (\C^{n+2})^{\odot 2} \otimes (\C^{n+2})^{\wedge 2} \otimes V_{(0,\lambda_2)} \rightarrow ((\C^{n+2})^{\wedge 2} )^{\otimes 2} \otimes \C \otimes V_{(0,\lambda_2)}$$
given by
$$ v_1 \otimes v_2 v_3 \otimes (v_4 \wedge v_5) \otimes v_6 \mapsto \left( (v_1 \wedge v_2) \otimes v_3 + (v_1 \wedge v_3) \otimes v_2 \right) \otimes v_4 \wedge v_5 \otimes v_6,$$
\end{comment}
$$Q'': \C^{n+2} \otimes (\C^{n+2})^{\wedge 2} \otimes (\C^{n+2})^{\odot 2} \otimes V_{(0,\lambda_2)} \rightarrow ((\C^{n+2})^{\wedge 2} )^{\otimes 2} \otimes \C \otimes V_{(0,\lambda_2)}$$
given by
$$ v_1 \otimes (v_2 \wedge v_3) \otimes v_4 v_5 \otimes v_6 \mapsto \left( (v_1 \wedge v_4) \otimes v_5 + (v_1 \wedge v_5) \otimes v_4 \right) \otimes v_2 \wedge v_3 \otimes v_6.$$
We get the following chain of equalities, which holds up to the same constant; as we are only interested in whether the sign of two projections will be the same, we will ignore this constant.
\begin{flalign*}
& P_{(1, \lambda_2+2)}\left(  e_{\pm} \otimes e_{\pm} \otimes P_{(3,\lambda_2)} \left( e_{\mp} \otimes e_{\mp} \otimes v^{1}_{(1,\lambda_2)} \right) \right)
\\ & = P_{(1, \lambda_2+2)}\left(  e_{\pm} \otimes P_{(2,\lambda_2+1)} \left( e_{\pm} \otimes P_{(3,\lambda_2)} \left( e_{\mp} \otimes e_{\mp} \otimes v^{1}_{(1,\lambda_2)} \right) \right) \right)
\\ & = P_{(1, \lambda_2+2)}\left(  e_{\pm} \otimes P_{(2,\lambda_2+1)} Q' \left( e_{\pm} \otimes  e_+ e_{\mp}^{\odot 2} \otimes v^{0}_{(0,\lambda_2)} \right) \right) = P_{(1, \lambda_2+2)} Q'' \left(  e_{\pm} \otimes Q' \left( e_{\pm} \otimes  e_+ e_{\mp}^{\odot 2} \otimes v^{0}_{(0,\lambda_2)} \right) \right).
\end{flalign*}
Thus we see that
\begin{flalign*}
& P_{(1, \lambda_2+2)}\left(  e_+ \otimes e_+ \otimes P_{(3,\lambda_2)} \left( e_- \otimes e_- \otimes v^{1}_{(1,\lambda_2)} \right) \right) = 2 P_{(1, \lambda_2+2)} Q'' \left(  e_+ \otimes  (e_+ \wedge e_-) \otimes e_+ e_- \otimes v^0_{(0,\lambda_2)} \right) 
\\ & = 2 P_{(1, \lambda_2+2)}\left( (e_+ \wedge e_-) \otimes e_+ \otimes (e_+ \wedge e_-) \otimes v^0_{(0,\lambda_2)} \right),
\end{flalign*}
while
\begin{flalign*}
& P_{(1,\lambda_2+2)}\left( e_- \otimes e_- \otimes P_{(3, \lambda_2)} \left( e_+ \otimes e_+ \otimes v^{1} \right) \right) = 3 P_{(1,\lambda_2+2)} Q'' \left( e_- \otimes (e_- \wedge e_+) \otimes e_+^{\odot 2} \otimes v^0_{(0,\lambda_2)} \right)
\\ & = 6 P_{(1,\lambda_2+2)}\left( (e_- \wedge e_+) \otimes e_+ \otimes (e_- \wedge e_+) \otimes v^0_{(0,\lambda_2)} \right) = 6 P_{(1, \lambda_2+2)}\left( (e_+ \wedge e_-) \otimes e_+ \otimes (e_+ \wedge e_-) \otimes v^0_{(0,\lambda_2)} \right).
\end{flalign*}
We conclude that the signs are the same and
\begin{flalign*}
& T_{(3,\lambda_2)}^{(1, \lambda_2+2)}(\nu)T_{(1,\lambda_2)}^{(3,\lambda_2)}(\nu) v^{1,+}
\\ & = \left( (\nu_1 + 2 \lambda_2 + 3 + (-1)^{\lambda_2} 2 \frac{n+2}{n} \nu_2)(\nu_1 + 2 \lambda_2 - 1 - (-1)^{\lambda_2} 2  \frac{n+2}{n} \nu_2) + 3(\nu_1 + 2 \lambda_2 + 1)^2 \right) C_1 v^{1,+}
\end{flalign*}
for some coefficient $C_1 \in \C$. A similar argument works for $v^{1,-}$.
\end{proof}

\bibliographystyle{amsplain}
\bibliography{bibdb}

@article {Kab12,
	AUTHOR = {Kable, Anthony C.},
	TITLE = {Conformally Invariant Systems of Differential Equations on Flag Manifolds for $G_2$ and their {K}-Finite Solutions}}

@article {Shim1990,
    AUTHOR = {Shimeno, Nobukazu},
     TITLE = {Eigenspaces of invariant differential operators on a
              homogeneous line bundle on a {R}iemannian symmetric space},
   JOURNAL = {J. Fac. Sci. Univ. Tokyo Sect. IA Math.},
  FJOURNAL = {Journal of the Faculty of Science. University of Tokyo.
              Section IA. Mathematics},
    VOLUME = {37},
      YEAR = {1990},
    NUMBER = {1},
     PAGES = {201--234},
      ISSN = {0040-8980},
   MRCLASS = {22E30 (58G25)},
  MRNUMBER = {1049027},
MRREVIEWER = {Rebecca\ Herb},
}

@article {Vog84,
    AUTHOR = {Vogan, Jr., David A.},
     TITLE = {Unitarizability of certain series of representations},
   JOURNAL = {Ann. of Math. (2)},
  FJOURNAL = {Annals of Mathematics. Second Series},
    VOLUME = {120},
      YEAR = {1984},
    NUMBER = {1},
     PAGES = {141--187},
      ISSN = {0003-486X,1939-8980},
   MRCLASS = {22E47 (22E45)},
  MRNUMBER = {750719},
MRREVIEWER = {David\ Collingwood},
       DOI = {10.2307/2007074},
       URL = {https://doi.org/10.2307/2007074},
}

@incollection {KubOr25,
    AUTHOR = {Kubo, Toshihisa and \O rsted, Bent},
     TITLE = {Classification of {$K$}-type formulas for the {H}eisenberg
              ultrahyperbolic operator {$\square_s$} for
              {$\widetilde{SL}(3,\Bbb R)$} and tridiagonal determinants for
              local {H}eun functions},
 BOOKTITLE = {Symmetry in geometry and analysis. {V}ol. 2. {F}estschrift in
              honor of {T}oshiyuki {K}obayashi},
    SERIES = {Progr. Math.},
    VOLUME = {358},
     PAGES = {303--382},
 PUBLISHER = {Birkh\"auser/Springer, Singapore},
      YEAR = {[2025] \copyright 2025},
      ISBN = {978-981-97-7661-0; 978-981-97-7662-7},
   MRCLASS = {22E46 (05E10 17B10 33C05 33E30 35H10)},
  MRNUMBER = {4866995},
       DOI = {10.1007/978-981-97-7662-7\_8},
       URL = {https://doi.org/10.1007/978-981-97-7662-7_8},
}

@article {Tor83,
    AUTHOR = {Torasso, Pierre},
     TITLE = {Quantification g\'eom\'etrique, op\'erateurs d'entrelacement
              et repr\'esentations unitaires de {$(\widetilde {\rm
              SL})_3({\bf R})$}},
   JOURNAL = {Acta Math.},
  FJOURNAL = {Acta Mathematica},
    VOLUME = {150},
      YEAR = {1983},
    NUMBER = {3-4},
     PAGES = {153--242},
      ISSN = {0001-5962,1871-2509},
   MRCLASS = {22E46 (58F06)},
  MRNUMBER = {709141},
MRREVIEWER = {J.\ H.\ Rawnsley},
       DOI = {10.1007/BF02392971},
       URL = {https://doi.org/10.1007/BF02392971},
}

@article {FWZ22,
    AUTHOR = {Frahm, Jan and Weiske, Clemens and Zhang, Genkai},
     TITLE = {Heisenberg parabolically induced representations of
              {H}ermitian {L}ie groups, {P}art {I}: {I}ntertwining operators
              and {W}eyl transform},
   JOURNAL = {Adv. Math.},
  FJOURNAL = {Advances in Mathematics},
    VOLUME = {422},
      YEAR = {2023},
     PAGES = {Paper No. 109001, 56},
      ISSN = {0001-8708,1090-2082},
   MRCLASS = {22E46 (17B15 17B60 22D30 43A80 43A85)},
  MRNUMBER = {4575032},
MRREVIEWER = {Jacques\ Faraut},
       DOI = {10.1016/j.aim.2023.109001},
       URL = {https://doi.org/10.1016/j.aim.2023.109001},
}

@article {Fra22,
    AUTHOR = {Frahm, Jan},
     TITLE = {Conformally invariant differential operators on {H}eisenberg
              groups and minimal representations},
   JOURNAL = {M\'em. Soc. Math. Fr. (N.S.)},
  FJOURNAL = {M\'emoires de la Soci\'et\'e{} Math\'ematique de France.
              Nouvelle S\'erie},
    NUMBER = {180},
      YEAR = {2024},
     PAGES = {vi+139},
      ISSN = {0249-633X,2275-3230},
      ISBN = {978-2-85629-986-9},
   MRCLASS = {22E45 (17B08 22E25 22E46 43A30)},
  MRNUMBER = {4740041},
MRREVIEWER = {Jacques\ Faraut},
       DOI = {10.24033/msmf.488},
       URL = {https://doi.org/10.24033/msmf.488},
}

@book {Kna86,
    AUTHOR = {Knapp, Anthony W.},
     TITLE = {Representation theory of semisimple groups},
    SERIES = {Princeton Mathematical Series},
    VOLUME = {36},
      NOTE = {An overview based on examples},
 PUBLISHER = {Princeton University Press},
   ADDRESS = {Princeton, NJ},
      YEAR = {1986}
}

@article {BKZ08,
    AUTHOR = {Barchini, Leticia and Kable, Anthony C. and Zierau, Roger},
     TITLE = {Conformally invariant systems of differential equations and
              prehomogeneous vector spaces of {H}eisenberg parabolic type},
   JOURNAL = {Publ. Res. Inst. Math. Sci.},
  FJOURNAL = {Kyoto University. Research Institute for Mathematical
              Sciences. Publications},
    VOLUME = {44},
      YEAR = {2008},
    NUMBER = {3},
     PAGES = {749--835}
}

@phdthesis{dahl19,
    author = {Tue Thulesen Dahl},
    title = {Knapp-Stein operators and Fourier transformations},
    school = {Department of Mathematics, Aarhus University},
    year = {2019} 
}

@article {Tam23,
    AUTHOR = {Tamori, Hiroyoshi},
     TITLE = {Classification of irreducible {$(\mathfrak{g},\mathfrak{k})$}-modules
              associated to the ideals of minimal nilpotent orbits for
              simple {L}ie groups of type {$A$}},
   JOURNAL = {Int. Math. Res. Not. IMRN},
  FJOURNAL = {International Mathematics Research Notices. IMRN},
      YEAR = {2023},
    NUMBER = {5},
     PAGES = {3814--3847},
      ISSN = {1073-7928,1687-0247},
   MRCLASS = {17B08 (17B10 22E46)},
  MRNUMBER = {4565656},
MRREVIEWER = {Jan\ Frahm},
       DOI = {10.1093/imrn/rnab356},
       URL = {https://doi.org/10.1093/imrn/rnab356},
}

@article {Zha22,
    AUTHOR = {Zhang, Genkai},
     TITLE = {Principal series of {H}ermitian {L}ie groups induced from
              {H}eisenberg parabolic subgroups},
   JOURNAL = {J. Funct. Anal.},
  FJOURNAL = {Journal of Functional Analysis},
    VOLUME = {282},
      YEAR = {2022},
    NUMBER = {8},
     PAGES = {Paper No. 109399, 39}
}

@article {HT93,
	AUTHOR = {Howe, Roger and Tan, Eng-Chye},
	TITLE = {Homogeneous functions on light cones: the infinitesimal structure of some degenerate principal series representations},
	JOURNAL = {Bull. Amer. Math. Soc.},
	FJOURNAL = {Bulletin of the American Mathematical Society},
	VOLUME = {28},
	YEAR = {1993},
	PAGES = {1--74}
}

@article {Sch84,
	AUTHOR = {Schlichtkrull, Henrik},
	TITLE = {One-dimensional {$K$}-types in finite-dimensional
	representations of semisimple {L}ie groups: a generalization
	of {H}elgason's theorem},
	JOURNAL = {Math. Scand.},
	FJOURNAL = {Mathematica Scandinavica},
	VOLUME = {54},
	YEAR = {1984},
	NUMBER = {2},
	PAGES = {279--294}
}

@article {Joh90,
    AUTHOR = {Johnson, Kenneth D.},
     TITLE = {Degenerate principal series and compact groups},
   JOURNAL = {Math. Ann.},
  FJOURNAL = {Mathematische Annalen},
    VOLUME = {287},
      YEAR = {1990},
    NUMBER = {4},
     PAGES = {703--718}
}

@article {JW77,
	AUTHOR = {Johnson, Kenneth D. and Wallach, Nolan R.},
	TITLE = {Composition series and intertwining operators for the
	spherical principal series. {I}},
	JOURNAL = {Trans. Amer. Math. Soc.},
	FJOURNAL = {Transactions of the American Mathematical Society},
	VOLUME = {229},
	YEAR = {1977},
	PAGES = {137--173}
}

@book {Kna02,
    AUTHOR = {Knapp, Anthony W.},
     TITLE = {Lie groups beyond an introduction},
    SERIES = {Progress in Mathematics},
    VOLUME = {140},
   EDITION = {Second},
 PUBLISHER = {Birkh\"auser Boston, Inc., Boston, MA},
      YEAR = {2002},
     PAGES = {xviii+812},
      ISBN = {0-8176-4259-5},
   MRCLASS = {22-01},
  MRNUMBER = {1920389},
}

@article {Vret76,
    AUTHOR = {Vretare, Lars},
     TITLE = {Elementary spherical functions on symmetric spaces},
   JOURNAL = {Math. Scand.},
  FJOURNAL = {Mathematica Scandinavica},
    VOLUME = {39},
      YEAR = {1976},
    NUMBER = {2},
     PAGES = {343--358},
      ISSN = {0025-5521,1903-1807},
   MRCLASS = {43A90 (22E45)},
  MRNUMBER = {447979},
MRREVIEWER = {Mitsuo\ Sugiura},
       DOI = {10.7146/math.scand.a-11667},
       URL = {https://doi.org/10.7146/math.scand.a-11667},
}

@incollection {How08,
    AUTHOR = {Howe, Roger E.},
     TITLE = {Some recent applications of induced representations},
 BOOKTITLE = {Group representations, ergodic theory, and mathematical
              physics: a tribute to {G}eorge {W}. {M}ackey},
    SERIES = {Contemp. Math.},
    VOLUME = {449},
     PAGES = {173--191},
 PUBLISHER = {Amer. Math. Soc., Providence, RI},
      YEAR = {2008},
      ISBN = {978-0-8218-4225-6},
   MRCLASS = {22D30 (22D15 22E27 22E50 81P68)},
  MRNUMBER = {2391804},
MRREVIEWER = {Hadi\ Salmasian},
       DOI = {10.1090/conm/449/08712},
       URL = {https://doi.org/10.1090/conm/449/08712},
}

@article {SpehVog80,
    AUTHOR = {Speh, Birgit and Vogan, Jr., David A.},
     TITLE = {Reducibility of generalized principal series representations},
   JOURNAL = {Acta Math.},
  FJOURNAL = {Acta Mathematica},
    VOLUME = {145},
      YEAR = {1980},
    NUMBER = {3-4},
     PAGES = {227--299},
      ISSN = {0001-5962,1871-2509},
   MRCLASS = {22E46},
  MRNUMBER = {590291},
MRREVIEWER = {A.\ U.\ Klimyk},
       DOI = {10.1007/BF02414191},
       URL = {https://doi.org/10.1007/BF02414191},
}

@article {HoweLee99,
    AUTHOR = {Howe, Roger and Lee, Soo Teck},
     TITLE = {Degenerate principal series representations of {${\rm
              GL}_n(\bold C)$} and {${\rm GL}_n(\bold R)$}},
   JOURNAL = {J. Funct. Anal.},
  FJOURNAL = {Journal of Functional Analysis},
    VOLUME = {166},
      YEAR = {1999},
    NUMBER = {2},
     PAGES = {244--309},
      ISSN = {0022-1236,1096-0783},
   MRCLASS = {22E46},
  MRNUMBER = {1707754},
MRREVIEWER = {Leticia\ Barchini},
       DOI = {10.1006/jfan.1999.3427},
       URL = {https://doi.org/10.1006/jfan.1999.3427},
}

@article {Speh81,
    AUTHOR = {Speh, Birgit},
     TITLE = {The unitary dual of {${\rm Gl}(3,\,{\bf R})$}\ and {${\rm
              Gl}(4,\,{\bf R})$}},
   JOURNAL = {Math. Ann.},
  FJOURNAL = {Mathematische Annalen},
    VOLUME = {258},
      YEAR = {1981/82},
    NUMBER = {2},
     PAGES = {113--133},
      ISSN = {0025-5831,1432-1807},
   MRCLASS = {22E46},
  MRNUMBER = {641819},
MRREVIEWER = {Mogens\ Flensted-Jensen},
       DOI = {10.1007/BF01450529},
       URL = {https://doi.org/10.1007/BF01450529},
}

@article {BOO96,
    AUTHOR = {Branson, Thomas and \'Olafsson, Gestur and Ørsted, Bent},
     TITLE = {Spectrum generating operators and intertwining operators for
              representations induced from a maximal parabolic subgroup},
   JOURNAL = {J. Funct. Anal.},
  FJOURNAL = {Journal of Functional Analysis},
    VOLUME = {135},
      YEAR = {1996},
    NUMBER = {1},
     PAGES = {163--205},
      ISSN = {0022-1236,1096-0783},
   MRCLASS = {22E46 (22E47)},
  MRNUMBER = {1367629},
MRREVIEWER = {Kenneth\ Johnson},
       DOI = {10.1006/jfan.1996.0008},
       URL = {https://doi.org/10.1006/jfan.1996.0008},
}

@article {Zha23,
AUTHOR = {Zhang, Genkai},
TITLE = {Principal series of quaternionic and real split exceptional {L}ie groups induced from {H}eisenberg parabolic subgroups},
YEAR = {2023},
NOTE = {preprint, available at \href{https://arxiv.org/abs/2309.05980}{arXiv:2309.05980}}
}

@article {FHWZ26,
AUTHOR = {Frahm, Jan and van Haastrecht, Robin and Weiske, Clemens and Zhang, Genkai},
TITLE = {On the degenerate principal series of $G_{2(2)}$ induced from a Heisenberg parabolic subgroup},
YEAR = {2026},
NOTE = {to appear in Represent. Theory, available at available at \href{https://arxiv.org/abs/2601.20344}{arXiv:2601.20344}}
}

@article {Vog86,
    AUTHOR = {Vogan, Jr., David A.},
     TITLE = {The unitary dual of {${\rm GL}(n)$} over an {A}rchimedean
              field},
   JOURNAL = {Invent. Math.},
  FJOURNAL = {Inventiones Mathematicae},
    VOLUME = {83},
      YEAR = {1986},
    NUMBER = {3},
     PAGES = {449--505},
      ISSN = {0020-9910,1432-1297},
   MRCLASS = {22E46},
  MRNUMBER = {827363},
MRREVIEWER = {D.\ Mili\v ci\'c},
       DOI = {10.1007/BF01394418},
       URL = {https://doi.org/10.1007/BF01394418},
}

@article {Kab12sln2,
    AUTHOR = {Kable, Anthony C.},
     TITLE = {The {H}eisenberg ultrahyperbolic equation: {$K$}-finite and
              polynomial solutions},
   JOURNAL = {Kyoto J. Math.},
  FJOURNAL = {Kyoto Journal of Mathematics},
    VOLUME = {52},
      YEAR = {2012},
    NUMBER = {4},
     PAGES = {839--894},
      ISSN = {2156-2261,2154-3321},
   MRCLASS = {22E30 (35R03)},
  MRNUMBER = {2998915},
MRREVIEWER = {Nguyen\ Lam},
       DOI = {10.1215/21562261-1728911},
       URL = {https://doi.org/10.1215/21562261-1728911},
}

@book {Bour02,
    AUTHOR = {Bourbaki, Nicolas},
     TITLE = {Lie groups and {L}ie algebras. {C}hapters 4--6},
    SERIES = {Elements of Mathematics (Berlin)},
      NOTE = {Translated from the 1968 French original by Andrew Pressley},
 PUBLISHER = {Springer-Verlag, Berlin},
      YEAR = {2002},
     PAGES = {xii+300},
      ISBN = {3-540-42650-7},
   MRCLASS = {17-01 (00A05 20E42 20F55 22-01)},
  MRNUMBER = {1890629},
       DOI = {10.1007/978-3-540-89394-3},
       URL = {https://doi.org/10.1007/978-3-540-89394-3},
}

@incollection {FWZ25,
    AUTHOR = {Frahm, Jan and Weiske, Clemens and Zhang, Genkai},
     TITLE = {Heisenberg parabolically induced representations of
              {H}ermitian {L}ie groups, {P}art {II}: {N}ext-to-minimal
              representations and branching rules},
 BOOKTITLE = {Symmetry in geometry and analysis. {V}ol. 2. {F}estschrift in
              honor of {T}oshiyuki {K}obayashi},
    SERIES = {Progr. Math.},
    VOLUME = {358},
     PAGES = {197--226},
 PUBLISHER = {Birkh\"auser/Springer, Singapore},
      YEAR = {[2025] \copyright 2025},
      ISBN = {978-981-97-7661-0; 978-981-97-7662-7},
   MRCLASS = {22E46 (17B15 17B60 22D30 43A65)},
  MRNUMBER = {4866993},
MRREVIEWER = {Jacques\ Faraut},
       DOI = {10.1007/978-981-97-7662-7\_6},
       URL = {https://doi.org/10.1007/978-981-97-7662-7_6},
}

@article {Wei03,
    AUTHOR = {Weissman, Martin H.},
     TITLE = {The {F}ourier-{J}acobi map and small representations},
   JOURNAL = {Represent. Theory},
  FJOURNAL = {Representation Theory. An Electronic Journal of the American
              Mathematical Society},
    VOLUME = {7},
      YEAR = {2003},
     PAGES = {275--299},
      ISSN = {1088-4165},
   MRCLASS = {22E50 (11F70)},
  MRNUMBER = {1993361},
MRREVIEWER = {Anthony\ C.\ Kable},
       DOI = {10.1090/S1088-4165-03-00197-3},
       URL = {https://doi.org/10.1090/S1088-4165-03-00197-3},
}

@incollection {KazhSav90,
    AUTHOR = {Kazhdan, D. and Savin, G.},
     TITLE = {The smallest representation of simply laced groups},
 BOOKTITLE = {Festschrift in honor of {I}. {I}. {P}iatetski-{S}hapiro on the
              occasion of his sixtieth birthday, {P}art {I} ({R}amat {A}viv,
              1989)},
    SERIES = {Israel Math. Conf. Proc.},
    VOLUME = {2},
     PAGES = {209--223},
 PUBLISHER = {Weizmann, Jerusalem},
      YEAR = {1990},
   MRCLASS = {22E50},
  MRNUMBER = {1159103},
MRREVIEWER = {David\ Joyner},
}

@article {Kor65,
    AUTHOR = {Kor\'anyi, Adam and Wolf, Joseph A.},
     TITLE = {Realization of hermitian symmetric spaces as generalized
              half-planes},
   JOURNAL = {Ann. of Math. (2)},
  FJOURNAL = {Annals of Mathematics. Second Series},
    VOLUME = {81},
      YEAR = {1965},
     PAGES = {265--288},
      ISSN = {0003-486X},
   MRCLASS = {32.32 (22.70)},
  MRNUMBER = {174787},
MRREVIEWER = {Hugo\ Rossi},
       DOI = {10.2307/1970616},
       URL = {https://doi.org/10.2307/1970616},
}

\end{document}